%% file: main_arXiv.tex
\documentclass{article}
\usepackage{arxiv}

\usepackage[utf8]{inputenc}

\usepackage{bm}
\usepackage{amsmath,amsfonts}
\usepackage{mathrsfs}
\usepackage{subcaption}
\usepackage[colorlinks=true, pdfstartview=FitV, linkcolor=black, citecolor=black, urlcolor=black]{hyperref}
\usepackage{footnpag}			      	% make footnote symbols restart on each page

\usepackage{booktabs}
\usepackage{csvsimple}
\usepackage[acronym]{glossaries}
\makeglossaries
\usepackage{mathtools}
\usepackage{microtype}
\usepackage{physics}
\usepackage[capitalize]{cleveref}
\usepackage{siunitx}
\usepackage{rotating}
\usepackage[svgnames]{xcolor}

\hypersetup{colorlinks=true,allcolors=RoyalBlue}
\glsdisablehyper

\graphicspath{{./figs/}}

\newcommand{\da}[1]{\left[#1\right]}
\newcommand{\dda}[1]{\var#1}
\newcommand{\vdda}[1]{\var\bm{#1}}
\newcommand{\ty}[3]{\mathcal{T}^{(#1)}_{#2}\left(#3\right)}

\newcommand{\R}{\mathbb R}
\newcommand{\N}{\mathbb N}
\newcommand{\M}{\mathscr M}
\newcommand{\X}{\mathscr X}

\newtheorem{lemma}{Lemma}
\newtheorem{theorem}{Theorem}
\newenvironment{proof}{\noindent {\bf Proof:}}{\hfill $\Box$}

\newcommand\extrafootertext[1]{%
    \bgroup
    \renewcommand\thefootnote{\fnsymbol{footnote}}%
    \renewcommand\thempfootnote{\fnsymbol{mpfootnote}}%
    \footnotetext[0]{#1}%
    \egroup
}

\title{POLYNOMIAL-BASED SOLUTIONS TO TARGETING PROBLEMS FOR ONBOARD APPLICATIONS}

\author{
  Adam Evans\thanks{Postdoctoral Fellow, Te P\={u}naha \={A}tea -- Space Institute} \\
  The University of Auckland, Auckland, New Zealand.
  \And
  Alberto Foss\`{a}\thanks{Postdoctoral Fellow, Oden Institute for Computational Engineering and Sciences} \\
  The University of Texas at Austin, Austin, TX.
  \And
  Roberto Armellin\thanks{Professor, Te P\={u}naha \={A}tea -- Space Institute}\\
  The University of Auckland, Auckland, New Zealand.
  \And
  Didier Henrion\thanks{Senior Researcher, LAAS-CNRS}\\
  Universit\'{e} de Toulouse, Toulouse, France.
  \And
  Renato Zanetti\thanks{Associate Professor, Department of Aerospace Engineering and Engineering Mechanics}\\
  The University of Texas at Austin, Austin, TX.
}

\date{}

\renewcommand{\headeright}{}
\renewcommand{\undertitle}{}
\renewcommand{\shorttitle}{POLYNOMIAL-BASED SOLUTIONS TO TARGETING PROBLEMS FOR ONBOARD APPLICATIONS}

\hypersetup{
pdftitle={POLYNOMIAL-BASED SOLUTIONS TO TARGETING PROBLEMS FOR ONBOARD APPLICATIONS},
pdfsubject={q-bio.NC, q-bio.QM},
pdfauthor={Adam Evans, Alberto Fossa, Roberto Armellin, Didier Henrion, Renato Zanetti},
}

\begin{document}

\include{Journal/acronyms}

\maketitle

\begin{abstract}
This paper solves the targeting problem focusing on accuracy, computational efficiency, and reliability. The trajectory optimization problem is first recast as a \gls{pop} by leveraging differential algebra to compute high-order Taylor expansions of the nonlinear dynamics and constraints. Moment-\gls{sos} optimization is then utilized to solve this \gls{pop}. A convex formulation based on a second-order expansion of the dynamics is also proposed. For impulsive targeting, the moment-\gls{sos} and convex approaches are compared against traditional \gls{nlp} solvers and map inversion techniques. Results indicate that the moment-\gls{sos} approach provides solutions as accurate as traditional \gls{nlp}, but with the critical advantage of guaranteeing convergence to the global optimum under mild assumptions. Furthermore, the method excels at handling large maneuvers and long propagation times, conditions in which standard linear approximations rapidly degrade. To demonstrate its versatility, the methodology is extended to a continuous low-thrust \gls{sk} scenario in the Earth-Moon Circular Restricted Three-Body Problem. The algorithm's performance is then evaluated in the presence of significant state errors. The ability to directly handle non-convex constraints and recast complex, nonlinear dynamics into formulations with reliable convergence properties makes the moment-\gls{sos} approach suitable for autonomous onboard applications.
\end{abstract}
\glsresetall

\section{Introduction}

As spacecraft operations evolve towards complete onboard autonomy, the need for efficient and robust trajectory re-planning has become increasingly important. However, the deployment of fully autonomous systems is constrained by limited onboard computational resources and the requirement for guaranteed solutions to complex orbital problems. This challenge is particularly evident in trajectory optimization, which traditionally relies on the numerical solution of large-scale \gls{nlp} problems. Even in fundamental orbital regimes, the highly nonlinear nature of the dynamics complicates the targeting and \gls{sk} problem, often resulting in cost landscapes characterized by multiple local minima. These complexities are further compounded in multi-body environments, such as the \gls{cr3bp}, where significant dynamical perturbations can render traditional first-order approximations insufficient. Furthermore, the inclusion of low-thrust propulsion systems introduces additional dimensionality and sensitivity, often necessitating the consideration of long time horizons and making the resulting optimal control problems notoriously difficult to solve.

To address these challenges, various strategies have been explored in the literature, each offering a trade-off between modeling fidelity and computational tractability. Traditional target point orbit maintenance strategies often rely on differential correction or linear-based control~\cite{Howell1993, farquhar1970control,howell1995station,gomez1998station}. While computationally efficient, these methods are fundamentally restricted by their reliance on first-order sensitivities, which limits their operational envelope to small maneuvers and short propagation intervals where linear assumptions hold. To expand this region of convergence, higher-order methods utilizing \gls{da} have been proposed for nonlinear targeting via map inversion~\cite{Fu2020, Evans2024_1, Evans2024_2}. However, while map inversion effectively captures the high-order flow of the dynamics, it typically lacks the ability to directly incorporate complex path constraints or navigate landscapes characterized by multiple local minima. Alternatively, convex optimization~\cite{Boyd2004} has gained traction due to its polynomial-time complexity and global convergence properties. However, its primary drawback remains that most astrodynamics problems are intrinsically non-convex. While this is often addressed by solving a sequence of relaxed convex sub-problems---known as successive convex optimization~\cite{Mao2016}---the formal convergence guarantees are often lost, and the method may struggle when the convexification process is inaccurate or faced with complex non-convex constraints.

Consequently, there remains a need for an optimization framework that occupies the space between the flexibility of \glspl{nlp} and the efficiency of convex methods. One such class of problems that has received significant recent attention is that of \glspl{pop}~\cite{Lasserre2001,Henrion2005,Lasserre2009}. Although not yet common in astrodynamics, \glspl{pop} allow for a more versatile formulation of the trajectory optimization problem while providing several advantages over traditional \gls{nlp} solvers. Specifically, dedicated algorithms such as moment-\gls{sos} optimization can solve \glspl{pop} efficiently while guaranteeing, under mild assumptions, that the global optimum is found. This provides a mathematically rigorous path to handling the nonlinearities of complex environments and non-convex operational constraints, while maintaining the robustness required for autonomous onboard applications.

This work leverages these advanced optimization techniques to solve a variety of spacecraft targeting and \gls{sk} problems. The primary contributions of this paper are as follows. First, complex targeting and \gls{sk} problems are systematically recast as \glspl{pop} by leveraging high-order \gls{da} Taylor expansions of the nonlinear dynamics and constraints. Second, for scenarios where a second-order expansion of the dynamics provides sufficient accuracy, an efficient convex relaxation of the original nonlinear problem is proposed. Finally, it is demonstrated that moment-\gls{sos} optimization effectively solves these formulations, offering significant operational benefits over standard solvers. Specifically, the approach provides guaranteed convergence to the global optimum under certain assumptions and exhibits superior robustness when handling large maneuvers, long propagation times, and non-convex constraints. This methodology is shown to bridge the gap between high-fidelity physical modeling and reliable numerical convergence across both impulsive and continuous-thrust regimes.

The remainder of this paper is organized as follows. \Cref{s:ImpulsiveTargetingProblem} formulates the impulsive targeting problem as an \gls{nlp} problem. \Cref{s:POP} subsequently recasts the general \gls{nlp} problem as a \gls{pop}. The second-order convex relaxation is proposed in \cref{s:ConvexOpt}, where techniques to improve numerical conditioning are discussed. Two numerical examples then evaluate these methods for impulsive targeting in \cref{s:ImpulsiveTargetingProblemApplications}: a classical Keplerian case benchmarking the formulations against the original \gls{nlp} solution, and a target point approach in the Earth-Moon \gls{cr3bp} comparing moment-\gls{sos} against linear approximations and map inversion. Finally, \cref{s:LTTargetingProblem} extends the methodology to a continuous low-thrust \gls{sk} scenario in the Earth-Moon \gls{cr3bp} to demonstrate its versatility and robustness against significant state errors, otherwise causing the spacecraft to rapidly diverge from the nominal orbit. Concluding remarks and future research directions are provided in \cref{s:Conclusion}.

\section{Impulsive Targeting Problem}\label{s:ImpulsiveTargetingProblem}
Consider the impulsive targeting problem where the objective is to minimize the initial impulsive maneuver required to steer the spacecraft from a given initial state to a desired final state in a fixed time interval. This problem can be formulated as follows:
\begin{equation}
    \min_{\Delta\bm{v}} \norm{\Delta\bm{v}}_2
    \label{eq:objective_function}
\end{equation}
subject to the equality constraints:
\begin{subequations}\label{eq:nonlinear_dynamics}
    \begin{align}
        \bm{x}(t_f) &= \bm{x}(t_0) + \int_{t_0}^{t_f} \bm{f}(\bm{x}(t),t)\dd{t}, \label{eq:nonlinear_dynamics_1} \\
        \bm{x}(t_0) &= \bm{x}_0^- + \begin{bmatrix} \bm{0}_{3\times 1} \\ \Delta\bm{v} \end{bmatrix}
        \label{eq:nonlinear_dynamics_2}
    \end{align}
\end{subequations}
and inequality constraint:
\begin{equation}
    \left[\bm{x}(t_f) - \overline{\bm{x}}_f\right]^T \bm{P}^{-1} \left[\bm{x}(t_f) - \overline{\bm{x}}_f\right] \leq d^2,
    \label{eq:mahalanobis_distance}
\end{equation}
where $\bm{f}(\bm{x}(t),t)$ are the ballistic dynamics, $\bm{x}_0^-\in\mathbb{R}^{6\times 1}$ is the initial state before the maneuver, $\overline{\bm{x}}_f\in\mathbb{R}^{6\times 1}$ is the target state, $t_0$ and $t_f$ are the initial and final times, $\Delta\bm{v}\in\mathbb{R}^{3\times 1}$ is the impulsive maneuver, $\bm{P}\in\mathbb{R}^{6\times 6}$ is the target covariance matrix, and $d\in\mathbb{R}$ is the target distance. \Cref{eq:mahalanobis_distance} is an upper bound on the squared Mahalanobis distance between the final state $\bm{x}(t)$ and a probability distribution with mean $\overline{\bm{x}}_f$ and covariance matrix $\bm{P}$.

\Cref{eq:objective_function,eq:nonlinear_dynamics,eq:mahalanobis_distance} define a \gls{nlp} problem that can be readily solved with dedicated algorithms such as \gls{ipopt}~\cite{Wachter2006}. As the objective function is not differentiable in $\Delta\bm{v}=\bm{0}$, \cref{eq:objective_function} is replaced by its square to avoid potential issues during the numerical solution of the \gls{nlp} problem. The objective thus becomes:
\begin{equation}
    \min_{\Delta\bm{v}} \norm{\Delta\bm{v}}_2^2
    \label{eq:objective_function_squared}
\end{equation}
subject to the equality constraints in \cref{eq:nonlinear_dynamics} and inequality constraint in \cref{eq:mahalanobis_distance}.

% ###############################################################################################
% ###############################################################################################

\section{Polynomial optimization problem}\label{s:POP}

Solving the \gls{nlp} problem defined by \cref{eq:objective_function_squared,eq:nonlinear_dynamics,eq:mahalanobis_distance} might become computationally intensive, as the dynamics in \cref{eq:nonlinear_dynamics_1} must be numerically integrated at each solver iteration. The proposed solution is to compute a high-order Taylor expansion of $\bm{x}(t_f)$ function of $\Delta\bm{v}$, and evaluate this expression in place of \cref{eq:nonlinear_dynamics_1}. As \cref{eq:nonlinear_dynamics_2,eq:mahalanobis_distance,eq:objective_function_squared} are already polynomial functions of $\Delta\bm{v}$, the original \gls{nlp} problem becomes a \gls{pop}. The latter can be then solved efficiently using the techniques described in the next section.

\subsection{Polynomial expansion of the constraints}

Consider the spacecraft state $\bm{x}$ defined as:
\begin{equation}
    \bm{x} = \left[ \bm{r}^T \quad \bm{v}^T \right]^T = \left[x \quad y \quad z \quad v_x \quad v_y \quad v_z\right]^T \in \mathbb{R}^{6\times 1},
    \label{eq:state_vector}
\end{equation}
and the impulsive maneuver $\Delta\bm{v}$ given by:
\begin{equation}
    \Delta\bm{v} = \left[ \Delta v_x \quad \Delta v_y \quad \Delta v_z \right]^T \in \mathbb{R}^{3\times 1}.
    \label{eq:maneuver_vector}
\end{equation}
A polynomial expansion of the final state can be obtained efficiently using \gls{da} techniques~\cite{Berz1999}. The impulsive maneuver is firstly initialized as:
\begin{equation}
    \da{\Delta\bm{v}} = \bm{0}_{3\times 1} + \vdda{v},
    \label{eq:delta_v_da}
\end{equation}
where the square brackets denote Taylor polynomials, and $\vdda{v}=\left[\dda{v}_x \quad \dda{v}_y \quad \dda{v}_z\right]^T$ are the three independent \gls{da} variables that represent an infinitesimal maneuver along the $x,y$ and $z$ directions, respectively. Then, evaluating \cref{eq:nonlinear_dynamics} in the \gls{da} framework results in the following expression for the final state:
\begin{equation}
    \da{\bm{x}(t_f)} = \ty{k}{\bm{x}(t_f)}{\vdda{v}},
    \label{eq:final_state_expansion}
\end{equation}
which is a vector of $k$-th order Taylor polynomials in $\vdda{v}$. Substituting \cref{eq:final_state_expansion} into the \gls{rhs} of \cref{eq:mahalanobis_distance} finally yields to the following approximation of the squared Mahalanobis distance:
\begin{equation}
    \begin{aligned}
        \da{d^2_M(t_f)} &= \left\{\da{\bm{x}(t_f)} - \overline{\bm{x}}_f\right\}^T \bm{P}^{-1} \left\{\da{\bm{x}(t_f)} - \overline{\bm{x}}_f\right\} \\
        &= \ty{k}{d^2_M(t_f)}{\vdda{v}},
    \end{aligned}
    \label{eq:mahalanobis_distance_expansion}
\end{equation}
which is again a $k$-th order Taylor polynomial in $\vdda{v}$. As the flow of the dynamics is embedded in \cref{eq:final_state_expansion}, \cref{eq:nonlinear_dynamics,eq:mahalanobis_distance} are replaced by the single inequality constraint:
\begin{equation}
    \ty{k}{d^2_M(t_f)}{\Delta\bm{v}} \leq d^2.
    \label{eq:constraint_pop}
\end{equation}
and the original \gls{nlp} problem is recast as the \gls{pop} of minimizing \cref{eq:objective_function_squared} subject to \cref{eq:constraint_pop}.

\subsection{Quadratic expansion of the constraints}

Consider the polynomial expansion of the final state given by \cref{eq:final_state_expansion}. Choosing $k=2$ and expanding terms yields
\begin{equation}
    \da{x_i(t_f)} = \sum_{\abs{\bm{\alpha}}\leq2} c_{i,\bm{\alpha}} \vdda{v}^{\bm{\alpha}} \qquad i\in[1,6],
    \label{eq:explicit_taylor_coefficients}
\end{equation}
where $\da{x_i(t_f)}$ are the six components of the \gls{da} vector $\da{\bm{x}(t_f)}$, $c_{i,\bm{\alpha}}$ are the coefficients of the monomials function of $\vdda{v}^{\bm{\alpha}}$, and $\bm{\alpha}\coloneqq(\alpha_j)\in\N^3$ is a multi-index with $\alpha_j$ corresponding to the exponent of $\dda{v}_j$ for $j=x,y,z$. \Cref{eq:explicit_taylor_coefficients} can then be rewritten in matrix form as
\begin{equation}
    \da{\bm{x}(t_f)} = \bm{c}_0 + \bm{A}\bm{z},
    \label{eq:final_state_expansion_matrix}
\end{equation}
where $\bm{c}_0\in\mathbb{R}^{6\times 1}$ is the vector of constant coefficients, $\bm{A}\in\mathbb{R}^{6\times 9}$ is the matrix of first and second order coefficients, and $\bm{z}$ is defined as
\begin{equation}
    \begin{aligned}
        \bm{z} &= \left[z_{\bm{\alpha}}\right]_{\abs{\bm{\alpha}}\leq 2} = \left[ z_{100} \quad z_{010} \quad z_{001} \quad z_{200} \quad z_{110} \quad z_{020} \quad z_{101} \quad z_{011} \quad z_{002} \right]^T \\
        &= \left[\vdda{v}^{\bm{\alpha}}\right]_{\abs{\bm{\alpha}}\leq 2} = \left[ \dda{v}_x \quad \dda{v}_y \quad \dda{v}_z \quad \dda{v}_x^2 \quad \dda{v}_x\dda{v}_y \quad \dda{v}_y^2 \quad \dda{v}_x\dda{v}_z \quad \dda{v}_y\dda{v}_z \quad \dda{v}_z^2 \right]^T.
    \end{aligned}
    \label{eq:optimization_variables}
\end{equation}
Substituting \cref{eq:final_state_expansion_matrix} into \cref{eq:mahalanobis_distance} yields
\begin{equation}
    \left[\bm{A}\bm{z} - \left(\overline{\bm{x}}_f - \bm{c}_0\right)\right]^T \bm{P}^{-1} \left[\bm{A}\bm{z} - \left(\overline{\bm{x}}_f - \bm{c}_0\right)\right] \leq d^2,
\end{equation}
which can be rewritten as
\begin{equation}
    \bm{z}^T\bm{Q}\bm{z} + \bm{c}^T\bm{z} + r \leq 0,
    \label{eq:quadratic_constraint}
\end{equation}
where
\begin{subequations}
    \begin{align}
        \bm{Q} &= \bm{A}^T\bm{P}^{-1}\bm{A}, \\
        \bm{c} &= -2\left(\overline{\bm{x}}_f - \bm{c}_0\right)\bm{P}^{-1}\bm{A}, \\
        r &= \left(\overline{\bm{x}}_f - \bm{c}_0\right)^T\bm{P}^{-1}\left(\overline{\bm{x}}_f - \bm{c}_0\right) - d^2.
    \end{align}
    \label{eq:quadratic_constraint_components}
\end{subequations}
\Cref{eq:quadratic_constraint} is a fourth-order approximation of the constraint on the squared Mahalanobis distance function of the \gls{da} variables $\vdda{v}$. It does however not coincide with \cref{eq:mahalanobis_distance_expansion} for $k=4$, as the dynamics in \cref{eq:explicit_taylor_coefficients,eq:final_state_expansion_matrix} are truncated at second order. Moreover, additional constraints must be enforced to ensure the consistency of the polynomial expansion according to \cref{eq:optimization_variables}:
\begin{equation}
    \begin{aligned}
        z_{100}^2 - z_{200} = 0 &\qquad z_{100} z_{010} - z_{110} = 0 \\
        z_{010}^2 - z_{020} = 0 &\qquad z_{100} z_{001} - z_{101} = 0 \\
        z_{001}^2 - z_{002} = 0 &\qquad z_{010} z_{001} - z_{011} = 0.
    \end{aligned}
    \label{eq:equality_constraints_nonconvex}
\end{equation}
\Cref{eq:quadratic_constraint_components,eq:equality_constraints_nonconvex} provides the basis for the convex relaxation of the targeting problem presented later in this paper.

% ###############################################################################################
% ###############################################################################################

\section{Moment-sum-of-squares optimization}\label{s:mom-SOS}

Although \glspl{pop} can be solved using generic \gls{nlp} solvers, these algorithms treat the objective and constraint functions as black-boxes, and do not exploit their polynomial structure. Moreover, most solvers are gradient-based, meaning that the provided solution might be only locally optimal. In contrast, moment-\gls{sos} optimization~\cite{Lasserre2001,Lasserre2009} exploits the polynomial structure of the problem and builds a hierarchy of convex problems that can be solved efficiently via \gls{sdp}~\cite{Lasserre2009} techniques. Their solutions provide increasingly tighter lower bounds of the cost function until the optimal value is attained. The optimal solution to the original \gls{pop} is then recovered from the underlying relaxation.

\subsection{Moment-SOS hierarchy}

The following paragraphs recall the basic ingredients of the moment-\gls{sos} hierarchy. Readers interested in further details can consult the lecture notes~\cite{Henrion2023} or the more technical and comprehensive monograph~\cite{Lasserre2009}.

Consider the \gls{pop}
\begin{equation}\label{eq:pop}
    p^* \coloneqq \min_{\bm{\chi} \in \X} p(\bm{\chi})
\end{equation}
where $p$ is a given multivariate polynomial of the indeterminates $\bm{\chi} \in \R^n$, and
\begin{equation}
    \X \coloneqq \left\{\bm{\chi} \in \R^n : g_k(\bm{\chi}) \geq 0,\: k=1,\ldots,m \right\}
\end{equation}
is a bounded set defined by a given vector $\bm{g}\coloneqq(g_k)_{k=1,\ldots,m}$ of polynomials. Since $\X$ is compact and $p$ is continuous, a minimizer exists. Yet, the \gls{pop} in \cref{eq:pop} is a difficult global optimization problem with potentially many local and global optima.

The first step consists in reformulating the \gls{pop} as a finite-dimensional \gls{lp} problem. Let $\mathbb R[\bm{\chi}]_d$ denote the vector space of polynomials of $\bm{\chi}$ of degree at most $d$. It can be indexed by $\N^n_d\coloneqq\{\bm{\alpha} \in \N^n : \sum_{i=1}^n \alpha_i \leq d\}$.
%
%of dimension the binomial coefficient $|\N^n_d|\coloneqq\frac{(n+d)!}{n!d!}$.
%
Let $\bm{b}(\bm{\chi})\coloneqq(b_{\bm{\alpha}}(\bm{\chi}))_{\bm{\alpha}\in \N^n_d}$ denote a basis for this space, with the convention that $b_{\bm{0}}(\bm{\chi})\coloneqq 1$, so that every element $p \in \R[\bm{\chi}]_d$ can be expressed as a linear combination
\begin{equation}\label{poly}
    p(\bm{\chi}) = \sum_{\bm{\alpha}\in \N^n_d} p_{\bm{\alpha}} b_{\bm{\alpha}}(\bm{\chi}). %= \bm{p}^T \bm{b}(\bm{\chi})
\end{equation}
%
% with coefficient vector\footnote{$\bm{p}^T$ denotes the row vector transpose of column vector $\bm{p}$.} $\bm{p}\coloneqq(p_{\bm{\alpha}})_{\bm{\alpha}\in \N^n_d} \in \R^{|\N^n_d|}$.
%
Now consider a linear functional on $\R[\bm{\chi}]_d$. In basis $\bm{b}$, such a functional can be represented with a vector $\bm{y}\coloneqq(y_{\bm{\alpha}})_{\bm{\alpha} \in \N^n_d}$ as follows:
\begin{equation}\label{riesz}
 \ell_{\bm{y}}(p) \coloneqq \sum_{\bm{\alpha}\in \N^n_d} p_{\bm{\alpha}} y_{\bm{\alpha}}.
\end{equation}
Informally, $\ell_{\bm{y}}$ linearizes polynomials, compare \cref{poly} and \cref{riesz}. Note that $\ell_{\bm{y}}(1)=1$.
%
% The space of vectors $\bm{y}$ is dual\footnote{The dual to a vector space is the set of its (bounded) linear functionals.} to $\R[\bm{\chi}]_d$, and we denote it by $\R[\bm{\chi}]^{\star}_d$.
%
Given any vector $\bm{\chi} \in \X$, note that for the specific choice $\bm{y} = \bm{b}(\bm{\chi})$, the linear functional $\ell_{\bm{y}}(p) = p(\bm{\chi})$ coincides with the point evaluation at $\bm{\chi}$, and the expression $p(\bm{\chi})$ (which is nonlinear in $\bm{\chi}$) can be replaced with the expression $\ell_{\bm{y}}(p)$ (which is linear in $\bm{y}$). The constraint
\begin{equation}\label{nonconvex}
    \bm{y} \in \bm{b}(\mathscr X)\coloneqq\{\bm{b}(\bm{\chi}) : \bm{\chi} \in \mathscr X\}
\end{equation}
is however nonconvex. Since $\ell_{\bm{y}}(p)$ is linear in $\bm{y}$, \cref{nonconvex} can be replaced with the convex constraint
\begin{equation}\label{convex}
    \bm{y} \in \M(\X)\coloneqq\mathrm{conv}\:\bm{b}(\mathscr X)
\end{equation}
where $\mathrm{conv}$ denotes the convex hull\footnote{The convex hull of a set $\mathscr A$ is the smallest closed convex set containing $\mathscr A$.}. Now consider the finite-dimensional \gls{lp} problem
\begin{equation}\label{eq:lp}
    p^*_M\coloneqq\min_{\bm{y} \in \M(\X)} \ell_{\bm{y}}(p).
\end{equation}
It turns out that solving the \gls{lp} in \cref{eq:lp} is equivalent to solving the original nonconvex \gls{pop} in \cref{eq:pop}:
\begin{lemma}
    $p^*=p^*_M$.
\end{lemma}
\begin{proof}
    On one hand, for any global minimizer $\bm{\chi}^*$ of \cref{eq:pop}, notice that $\bm{y}^*\coloneqq\bm{b}(\bm{\chi}^*) \in \M(\X)$ is admissible for \cref{eq:lp}, with value $\ell_{\bm{y}^*}(p)=p(\bm{\chi}^*)=p^* \geq p^*_M$. On the other hand, it holds that $p(\bm{\chi}) \geq p^*$ for all admissible $\bm{\chi} \in \X$, and hence $\ell_{\bm{y}}(p) \geq \ell_{\bm{y}}(p^*)=p^*$ for any admissible $\bm{y} \in \M(\X)$. In particular this is true for a minimizer, implying $p^*_M \geq p^*$.
\end{proof}

The difficulty in solving \cref{eq:lp} is however concentrated into the constraint set $\M(\X)$. Despite being finite-dimensional and convex, this set is difficult to manipulate. Just determining whether it contains a given vector can be challenging. The main idea behind the moment-\gls{sos} hierarchy consists of approximating $\M(\X)$ with a family of increasingly tight semidefinite representable\footnote{A set is semidefinite representable if it is the linear projection of a linear section of a finite-dimensional cone of positive semidefinite matrices.} convex sets. Optimization of linear functions on semidefinite representable sets is the subject of \gls{sdp}, and it can be done efficiently with interior-point algorithms~\cite{Ben-Tal2001}.

A hierarchy of semidefinite representable outer approximations of $\M(\X)$ can be constructed as follows. Let $g_0\coloneqq 1$, let $r_{\bm{g}}$ be the smallest integer larger than $\max_{k} \frac{\text{deg}\,g_k}{2}$, let $r \geq r_{\bm{g}}$ be any integer and, for each $k$, let $r_k$ be the smallest integer larger than $r-\frac{\text{deg}\,g_k}{2}$. Then define
\begin{equation}
    \begin{aligned}
        \M(\X)_r \coloneqq & \left\{(y_{\bm{\alpha}})_{\bm{\alpha} \in \N^n_d} : \ell_{\bm{y}}(1)=1,\:\ell_{\bm{y}}(q^2 g_k) \geq 0, \:\forall q \in \R[\bm{\chi}]_{r_k}, \: k=0,1,\ldots,m\right\} \\
        =& \left\{(y_{\bm{\alpha}})_{\bm{\alpha} \in \N^n_d} : \ell_{\bm{y}}(1)=1,\:\vb{M}_{r_k}(g_k \bm{y}) \succeq 0,\: k=0,1,\ldots,m\right\}
    \end{aligned}
\end{equation}
where the matrix $\vb{M}_{r_k}(g_k \bm{y})$ represents the quadratic form $q \mapsto \ell_{\bm{y}}(g_k q^2)$ in basis $\bm{b}$, i.e.
\begin{equation}
    \vb{M}_{r_k}(g_k \bm{y})\coloneqq \ell_{\bm{y}}(g_k b_{\bm{\alpha}_1} b_{\bm{\alpha}_2})_{\bm{\alpha}_1,\bm{\alpha}_2 \in \N^n_{r_k}}.
\end{equation}
The notation reflects the fact that the matrix depends linearly on $g_k$ (for given $\bm{y}$) and also linearly on $\bm{y}$ (for given $g_k$). In particular, the positive semidefiniteness constraint $\vb{M}_{r_k}(g_k \bm{y}) \succeq 0$ is a \gls{lmi} in $\bm{y}$, which implies that $\M(\X)_r$ is semidefinite representable. When $k=0$, the matrix
\begin{equation}\label{mommat}
    \vb{M}_r(\bm{y})\coloneqq \ell_{\bm{y}}(\bm{b}\bm{b}^T)
\end{equation}
is called the moment matrix, where the linear functional acts entrywise on the rank-one matrix $\bm{b}\bm{b}^T$ with $T$ denoting the transpose. Given $\bm{\alpha} \in \N^n_d$, the scalar $y_{\bm{\alpha}}$ is called the pseudo-moment of degree $\bm{\alpha}$.

The following result, a consequence of the so-called Putinar Positivstellensatz, states that the convex semidefinite representable sets defined above are nested outer approximations converging (up to taking the closure) to $\M(\X)$:
\begin{lemma}
    $\M(\X)_r \supset \M(\X)_{r+1} \supset \cdots \supset \overline{\M(\X) _{\infty}} = \M(\X).$
\end{lemma}

Now consider the following hierarchy of \gls{sdp} problems called moment relaxations, indexed by the relaxation order $r$:
\begin{equation}\label{eq:sdp}
    p^*_r\coloneqq\min_{\bm{y} \in \M(\X)_r}\ell_{\bm{y}}(p).
\end{equation}
This hierarchy generates a monotonically non-decreasing converging sequence of lower bounds~\cite{Lasserre2001} on the \gls{pop} in \cref{eq:pop}:
\begin{theorem}\label{momsos}
    $p^*_r \leq p^*_{r+1} \leq \cdots \leq p^*_{\infty} = p^*$.
\end{theorem}
Let $\bm{y}^*$ denote the vector of pseudo-moments obtained by solving the moment relaxation of \cref{eq:sdp} for some given $r$. The objective is to determine whether $p^*_r=p^*$. The first candidate for global optimality is the vector of first degree pseudo-moments:
\begin{lemma}\label{firstdegree}
    Let $\bm{\chi}^{\star} \coloneqq({y}^{\star}_{\bm{\alpha}})_{|\bm{\alpha}|=1}$. If $\bm{\chi}^{\star} \in {\mathscr X}$ and $p(\bm{\chi}^{\star})=p^{\star}_r$ then $p^{\star}_r=p^{\star}$.
\end{lemma}
\begin{proof}
    Every admissible vector for the \gls{pop} yields an upper bound on the value of the \gls{pop}. If this upper bound is a lower bound, it means that it is optimal.
\end{proof}

\noindent
Another useful property to check is whether the moment matrix has rank one:
\begin{lemma}\label{rankone}
    If $\rank{\vb{M}_r(\bm{y}^{\star})} = 1$ then $p^{\star}_r=p^{\star}$.
\end{lemma}
\begin{proof}
    From the definition of the moment matrix in \cref{mommat}, if $\rank{\vb{M}_r(\bm{y}^{\star})}=1$ for some $\bm{y}^{\star}$, then $\vb{M}_r(\bm{y}^{\star})=\bm{b}(\bm{\chi}^{\star})\bm{b}(\bm{\chi}^{\star})^T=\bm{y}^{\star}(\bm{y}^{\star})^T.$ Since $\bm{y}^{\star}=\bm{b}(\bm{\chi}^{\star})$ is admissible for the moment relaxation in \cref{eq:sdp}, it follows that for all $k=1,\ldots,m$, $\vb{M}_{r_k}(g_k\bm{y}^{\star})=g_k(\bm{\chi}^{\star})\bm{b}(\bm{\chi}^{\star})\bm{b}(\bm{\chi}^{\star})^T \succeq 0$ and hence $g_k(\bm{\chi}^{\star}) \geq 0$, which means that $\bm{\chi}^{\star} \in {\mathscr X}$ is admissible for the \gls{pop} in \cref{eq:pop}. Since the value of the moment relaxation $p^{\star}_r=p(\bm{\chi}^{\star})$ is a lower bound on the value of the \gls{pop}, i.e. $p^\star$, and $\bm{\chi}^{\star}$ is admissible for the \gls{pop}, it follows that $\bm{\chi}^{\star}$ is optimal and therefore $p^{\star}_r=p^{\star}$.
\end{proof}

\noindent
A more general result called flat extension states:
\begin{theorem}\label{flat}
    For some $r^* \geq \max(r_{\bm{g}},\frac{d}{2})$, if vector $\bm{y}^*$ is such that
    \begin{equation*}
        \rank{\vb{M}_{r^*-r_{\bm{g}}}(\bm{y}^*)} = \rank{\vb{M}_{r^*}(\bm{y}^*)}
    \end{equation*}
    then $p^*_{r^*}=p^*$.
\end{theorem}
Note that \cref{rankone} is a particular case of \cref{flat}, since if $\rank{\vb{M}_{r^*}(\bm{y}^*)}=1$ for some $r^*$ then $\rank{\vb{M}_r(\bm{y}^*)}=1$ for all $r$. Similarly, \cref{firstdegree} is a particular case of \cref{flat}, since in this case the vector $\bm{y}^*=((\bm{\chi}^*)^{\bm{\alpha}})_{\bm{\alpha}}$ is such that $\rank{\vb{M}_r(\bm{y}^*)}=1$ for all $r$.

If the rank condition of \cref{flat} is satisfied, then there are $s\coloneqq\rank{\vb{M}_{r^*}(\bm{y}^{\star})}$ weights $w_k \geq 0$, $\sum_{k=1}^s w_k = 1$ and points $\bm{\chi}^{\star}_k \in {\mathscr X}$ such that $\vb{M}_r(\bm{y}^{\star})= \sum_{k=1}^s w_k \bm{b}(\bm{\chi}^{\star}_k) \bm{b}(\bm{\chi}^{\star}_k)^T$ and $p^{\star}=p(\bm{\chi}^{\star}_k)$ for all $k=1,\ldots,s$. Numerical linear algebra algorithms can then be applied on $\vb{M}_r(\bm{y}^{\star})$ to extract the points $\bm{\chi}^{\star}_k$, $k=1,\ldots,s$, which are all globally optimal for the \gls{pop} in \cref{eq:pop}.

In practice, for increasing values of the relaxation order $r$, the moment relaxations in \cref{eq:sdp} are constructed with the help of modeling software such as GloptiPoly~\cite{Henrion2008} or \texttt{MomentOpt.jl}\footnote{\url{https://github.com/lanl-ansi/MomentOpt.jl}}. These relaxations are solved numerically with an \gls{sdp} solver such as Mosek~\cite{Mosek2025} or \texttt{Hypatia.jl}~\cite{Coey2022}. The numerical solution is then post-processed to detect global optimality using the conditions of \cref{firstdegree,rankone,flat}, and to extract the global minimizers.

\subsection{Second-degree relaxation of the targeting problem}

Consider the \gls{pop} given by \cref{eq:objective_function_squared,eq:constraint_pop}. Setting the expansion order $k=4$ and using the notation of the previous paragraphs, the problem becomes:
\begin{equation}
    \min_{\bm{\chi}} p(\bm{\chi}) = \chi_1^2 + \chi_2^2 + \chi_3^2
    \label{eq:objective_relaxation}
\end{equation}
subject to
\begin{equation}
    g_1(\bm{\chi}) = d^2 - \ty{4}{d^2_M(t_f)}{\bm{\chi}} \geq 0,
    \label{eq:constraint_relaxation}
\end{equation}
where $\bm{\chi}\coloneqq\Delta\bm{v}$. \Cref{eq:objective_relaxation} is a second-order polynomial in $\bm{\chi}$, while \cref{eq:constraint_relaxation} is a fourth-order polynomial in the same variables. Therefore, $r_{\bm{g}}=2$ and the first relaxation has degree $r=2$. The moment vector in the monomial basis reads:
\begin{equation}
        \bm{y} = \left[y_{\bm{\alpha}}\right] = \left[\bm{\chi}^{\bm{\alpha}}\right], \qquad \bm{\alpha}\in\N_4^3
\end{equation}
and contains the $\binom{3+4}{3}=35$ moments from $y_{000}=1$ to $y_{004}=\chi_3^4$. The \gls{pop} given by \cref{eq:objective_function_squared,eq:constraint_pop} is thus equivalent to the following \gls{mop}:
\begin{equation}
    \min_{\bm{y}} p_M(\bm{y}) = y_{200} + y_{020} + y_{002}
    \label{eq:objective_mop}
\end{equation}
subject to
\begin{subequations}\label{eq:constraints_mop}
    \begin{align}
        \vb{M}_2(\bm{y}) &= 
        \begin{bmatrix}
            y_{000} & \ldots & y_{002} \\
            \vdots  & \ddots & \vdots  \\
            y_{002} & \ldots & y_{004} \\
        \end{bmatrix} \succeq 0 \label{eq:psd_constraint} \\
        \vb{M}_0(g_1\bm{y}) &= \bm{c}_{g_1}^T\bm{y}\\
        y_{000} &= 1,
    \end{align}
\end{subequations}
where $\vb{M}_2(\bm{y})$ is the moment matrix, which is a $10\times 10$ symmetric matrix constrained to be \gls{psd}, and $\bm{c}_{g_1}$ are the coefficients of \cref{eq:constraint_relaxation} corresponding to the monomials $\bm{\chi}^{\bm{\alpha}}$. The localizing matrix $\vb{M}_0(g_1\bm{y})$ is linear in the optimization variables $\bm{y}$, and the \gls{mop} can be efficiently solved by any \gls{sdp} solver.

% ###############################################################################################
% ###############################################################################################

\section{Convex optimization}\label{s:ConvexOpt}

Convex optimization has become very popular in astrodynamics, as it allows to solve large optimization problems efficiently, while guaranteeing that the global optimum is found~\cite{Boyd2004}. It consists in minimizing a convex function over a convex set, and it thus requires both the objective and the constraints to be convex functions in the optimization variables. In this section, a convex relaxation of the targeting problem is derived from the quadratic approximation of the constraint in \cref{eq:quadratic_constraint}.

\subsection{Convex relaxation of the problem}

Consider the optimization variables $\bm{z}$ defined in \cref{eq:optimization_variables}. After substituting these variables into \cref{eq:objective_function_squared}, the objective function becomes:
\begin{equation}
    \min_{\bm{z}} \quad z_{200} + z_{020} + z_{002}
    \label{eq:objective_convex}
\end{equation}
which is a linear function of $\bm{z}$. The constraints in \cref{eq:nonlinear_dynamics,eq:mahalanobis_distance} are then replaced by \cref{eq:quadratic_constraint,eq:equality_constraints_nonconvex}. As the latter are nonconvex functions of $\bm{z}$, they are relaxed into the following \gls{lmi}:
\begin{equation}
    \begin{bmatrix}
        1       & z_{100} & z_{010} & z_{001} \\
        z_{100} & z_{200} & z_{110} & z_{101} \\
        z_{010} & z_{110} & z_{020} & z_{011} \\
        z_{001} & z_{101} & z_{011} & z_{002}
    \end{bmatrix} \succeq 0,
    \label{eq:psd_constraint_convex}
\end{equation}
where the notation $\succeq 0$ requires the matrix on the \gls{lhs} to be \gls{psd}. This matrix corresponds to the moment matrix of degree one, denoted as $\vb{M}_1(\bm{y})$ in the previous section. Therefore, it must be singular with rank one for the equality constraints in \cref{eq:equality_constraints_nonconvex} to be satisfied. The convex relaxation of the original problem is finally given by \cref{eq:objective_convex,eq:quadratic_constraint,eq:psd_constraint_convex}.

\subsection{Numerically robust formulation}

The convex optimization problem presented above might be numerically ill-conditioned, as the matrix $\bm{Q}$ is not full rank, and its entries might span several orders of magnitude. To improve the numerical conditioning of the problem, a first step is to compute $\bm{Q}$ in a more robust way, i.e.:
\begin{subequations}
    \begin{align}
        \bm{B} &= \bm{D}\bm{A}, \\
        \bm{Q} &= \bm{B}^T\bm{B},
    \end{align}
\end{subequations}
where $\bm{D}$ is the square root of $\bm{P}^{-1}$ and $\bm{B}\in\mathbb{R}^{6\times 9}$. This formulation is used in all numerical examples presented in this paper.

For an even better conditioning, the inequality constraint in \cref{eq:quadratic_constraint} can be reformulated as follows. Firstly, an additional optimization variable $\beta\in\mathbb{R}$ is introduced, and the vector of optimization variables $\bm{z}$ is augmented to include $\beta$:
\begin{equation}
    \bm{\tilde{z}} = \begin{bmatrix} \bm{z} \\ \beta \end{bmatrix}.
\end{equation}
Then, \cref{eq:quadratic_constraint} is split into two constraints~\cite{Mosek2024}:
\begin{subequations}\label{eq:quadratic_constraint_split}
    \begin{align}
        \norm{\bm{R}\bm{z}}_2^2 &\leq \beta, \\
        \beta + \bm{c}^T\bm{z} + r &= 0.
    \end{align}
\end{subequations}
with $\bm{R}\in\mathbb{R}^{9\times 9}$ the upper triangular factor of the QR decomposition of $\bm{B}$ such that $\bm{Q}=\bm{R}^T\bm{R}$. The convex optimization problem is then defined by \cref{eq:objective_convex,eq:quadratic_constraint_split,eq:psd_constraint_convex}.

% ###############################################################################################
% ###############################################################################################

\section{Numerical applications of the impulsive targeting problem}\label{s:ImpulsiveTargetingProblemApplications}

This section presents two applications of the polynomial optimization techniques to the targeting problem in astrodynamics. The objective is to demonstrate their competitiveness compared to the solution of a \gls{nlp} and the use of polynomial map inversion techniques. All algorithms are implemented in Julia using different modeling frameworks and solvers. The polynomial expansions are computed with \texttt{GTPSA.jl}\footnote{\url{https://github.com/bmad-sim/GTPSA.jl}}, a Julia interface to the \gls{gtpsa} library~\cite{Deniau2015}. The original \gls{nlp} is modeled using \texttt{Optimization.jl}~\cite{Dixit2023}, as it integrates seamlessly with the \texttt{DifferentialEquations.jl}~\cite{Rackauckas2017} package used to solve \cref{eq:nonlinear_dynamics_1}. The \gls{mop} is modeled using the \texttt{MomentOpt.jl}\footnote{\url{https://github.com/lanl-ansi/MomentOpt.jl}}~\cite{Weisser2019} and \texttt{SumOfSquares.jl}~\cite{Legat2017,Weisser2019} packages, while the convex optimization problems are modeled using \texttt{Convex.jl}~\cite{Udell2014}. The \gls{nlp} problem is solved with \gls{ipopt}~\cite{Wachter2006,HSL2023}, while the \gls{sdp} and convex ones are solved with Mosek~\cite{Mosek2025}.

\subsection{Orbit Correction in the R2BP}
The first test case is an orbital correction scenario, as introduced in \cref{eq:objective_function,eq:nonlinear_dynamics,eq:mahalanobis_distance}. In this example the ballistic dynamics in \cref{eq:nonlinear_dynamics_1} is that of the \gls{r2bp}. After the initial maneuver, the following \glspl{eom} govern the motion of the spacecraft:
\begin{equation}
    \bm{f}(\bm{x}(t),t) = 
    \begin{bmatrix}
        \bm{v} \\
        -\frac{\mu}{\norm{\bm{r}}_2^3}\bm{r}
    \end{bmatrix},
\end{equation}
where $\bm{x}(t)$ is given by \cref{eq:state_vector} and $\mu$ is the standard gravitational parameter of the central body. The units are chosen such that $\mu=1$, the nominal orbit radius is $r=1$, and the nominal orbit period is $T=2\pi$. At time $t_0=0$, the spacecraft has deviated from its nominal orbit, and its state is given by:
\begin{equation}
    \begin{aligned}
        \bm{x}_0^- &= \overline{\bm{x}}_0 + \Delta\bm{x}_0 \\
        &= \left[ 1 \quad 0 \quad 0 \quad 0 \quad \sqrt{1/2} \quad \sqrt{1/2} \right]^T + \left[ 10^{-5} \quad 10^{-5} \quad 10^{-5} \quad 0 \quad 0 \quad 0 \right]^T.
    \end{aligned}
\end{equation}
The objective is to steer the spacecraft back to its nominal orbit in half orbit period, i.e. target the final state:
\begin{equation}
    \overline{\bm{x}}_f = \left[-1 \quad 0 \quad 0 \quad 0 \quad -\sqrt{1/2} \quad -\sqrt{1/2}\right]^T.
\end{equation}
at $t_f=\pi$. The constraint in \cref{eq:mahalanobis_distance} is thus enforced with target distance $d=\num{0.1}$ and covariance matrix $\bm{P}$ given by:
\begin{equation}
    \bm{P} = \mathrm{diag}\left(\left[\sigma_r^2 \quad \sigma_r^2 \quad \sigma_r^2 \quad \sigma_{v}^2 \quad \sigma_{v}^2 \quad \sigma_{v}^2\right]\right).
\end{equation}
where $\sigma_r=\num{0.1}$ and $\sigma_v=10^{-3}\sigma_r$ are the standard deviations on the components of the position and velocity vectors, respectively. To improve the numerical conditioning of the problem, the \gls{da} variables in \cref{eq:delta_v_da} are scaled by a constant coefficient $s=10^{-3}$, i.e.:
\begin{equation}
    \da{\Delta\bm{v}} = \bm{0}_{2\times 1} + s\cdot\vdda{v},
    \label{eq:delta_v_da_scaled}
\end{equation}
and the flow in \cref{eq:final_state_expansion} is computed using \cref{eq:delta_v_da_scaled} rather than \cref{eq:delta_v_da}. The optimal maneuver $\Delta\bm{v}^*$ is thus retrieved by multiplying the optimal solution to the corresponding problem by $s$. Four different problems are solved to compare the optimization techniques described in this paper. They are summarized in \cref{tab:summary}.

\begin{table}[ht]
    \centering
    \caption{Formulations of the optimization problems in the Keplerian dynamics scenario.}
    \label{tab:summary}
    \begin{tabular}{lllll}
        \toprule
        \# & Defining equations & Type & Framework & Solver \\
        \midrule
        A & \cref{eq:objective_function_squared,eq:nonlinear_dynamics,eq:mahalanobis_distance} & \gls{nlp} & \texttt{Optimization.jl} & \gls{ipopt} \\
        B & \cref{eq:objective_mop,eq:constraints_mop} & \gls{mop} & \texttt{MomentOpt.jl} & Mosek \\
        C & \cref{eq:objective_convex,eq:quadratic_constraint,eq:psd_constraint_convex} & convex & \texttt{Convex.jl} & Mosek \\
        D & \cref{eq:objective_convex,eq:quadratic_constraint_split,eq:psd_constraint_convex} & convex & \texttt{Convex.jl} & Mosek \\
        \bottomrule
    \end{tabular}
\end{table}

The solutions to these problems are given in \cref{tab:optimal_maneuvers}, in which both the components and magnitude of the optimal maneuvers are reported.

\begin{table}[ht]
    \caption{Optimal maneuvers in the Keplerian dynamics scenario.}
    \newcommand{\scinum}[1]{\tablenum[exponent-mode=scientific, round-precision=6, round-mode=places, table-format=-1.6e-1]{#1}}
    \begin{filecontents*}{Journal/data/results.csv}
pbm,dvx,dvy,dvz,dvn,dsq,eqr
A,-9.020764216763204e-6,-6.836395827453148e-6,-6.836395827453146e-6,1.3222964980429806e-5,9.982015346243767e-9,0.0
B,-9.020835349191185e-6,-6.836351612549729e-6,-6.8363516170980284e-6,1.3222967790939649e-5,-7.528310804610916e-10,0.0
C,-9.020777447546462e-6,-6.836389896715999e-6,-6.836389895873661e-6,1.3222967873608667e-5,-1.0855893874811562e-9,7.321633493169617e-12
D,-9.020624258083379e-6,-6.836491119427361e-6,-6.836490912918524e-6,1.3222967928341544e-5,-1.2025608980836822e-9,1.768767437745379e-12
\end{filecontents*}
\csvreader[
        head to column names,
        before reading = \begin{center},
        tabular = ccccc,
        table head = \toprule \# & {$\Delta v_x^*$} & {$\Delta v_y^*$} & {$\Delta v_z^*$} & {$\norm{\Delta\bm{v}^*}_2$} \\\midrule,
        table foot = \bottomrule,
        after reading = \end{center},
    ]{Journal/data/results.csv}{}{%
        \pbm &
        \scinum{\dvx} &
        \scinum{\dvy} &
        \scinum{\dvz} &
        \scinum{\dvn}
    }
    \label{tab:optimal_maneuvers}
\end{table}

The violation of the terminal constraint in \cref{eq:mahalanobis_distance} and of the equality constraints in \cref{eq:equality_constraints_nonconvex} are then assessed by computing the following quantities:
\begin{subequations}
    \begin{align}
        \epsilon_1 &= \left[\bm{x}^*(t_f) - \overline{\bm{x}}_f\right]^T \bm{P}^{-1} \left[\bm{x}^*(t_f) - \overline{\bm{x}}_f\right] - d^2 \\
        \epsilon_2 &= \norm{\bm{\zeta}^*\left(\bm{\zeta}^*\right)^T - \bm{Z}^*}_\infty
    \end{align}
\end{subequations}
where $\bm{x}^*(t_f)$ is computed from \cref{eq:nonlinear_dynamics} using the optimal maneuver $\Delta\bm{v}^*$, $\bm{\zeta}^* = \left[1 \quad \left[z^*_{\bm{\alpha}}\right]^T_{\abs{\bm{\alpha}}=1}\right]^T$, and $\bm{Z}^*$ is the \gls{lhs} of \cref{eq:psd_constraint_convex} evaluated at $\bm{z}^*$. A negative value of $\epsilon_1$ indicates that the final state is inside the target ellipsoid. Their values are reported in \cref{tab:constraints_violations}.

\begin{table}[ht]
    \caption{Constraints violations in the Keplerian dynamics scenario.}
    \newcommand{\scinum}[1]{\tablenum[exponent-mode=scientific, round-precision=6, round-mode=places, table-format=-1.6e-2]{#1}}
    \csvreader[
        head to column names,
        before reading = \begin{center},
        tabular = ccc,
        table head = \toprule \# & {$\epsilon_1$} & {$\epsilon_2$} \\\midrule,
        table foot = \bottomrule,
        after reading = \end{center},
    ]{Journal/data/results.csv}{}{%
        \pbm &
        \scinum{\dsq} &
        \scinum{\eqr}
    }
    \label{tab:constraints_violations}
\end{table}

\Cref{tab:optimal_maneuvers,tab:constraints_violations} demonstrate that all methods converge to the same solution within a very small margin of error. The \gls{nlp} and \gls{mop} formulations do not augment the free vector with the squares of the $\Delta\bm{v}$ components, and thus $\epsilon_2$ is identically zero for problems A and B. The solutions to problems C and D also satisfy \cref{eq:equality_constraints_nonconvex} within the tolerance of the underlying solver, meaning that the solution to the original problem is recovered with high accuracy. These results demonstrate that the proposed methods are a viable alternative to more widespread techniques, as they combine the accuracy of the \gls{nlp} formulation with the efficiency and convergence guarantee needed for onboard applications.

\subsection{Station-Keeping in the CR3BP}\label{ss:impulsive_station_keeping}

The optimization techniques are now applied to the \gls{tpa} to \gls{sk} in the \gls{cr3bp}~\cite{Fu2020}, in which the following \glspl{eom} govern the motion of the spacecraft:
\begin{equation}\label{eq:cr3bp_ballistic_dynamics}
    \bm{f}(\bm{x}(t),t) =
    \begin{bmatrix}
        \bm{v} \\[6pt]
        2\,\bm{\hat{z}}\!\times\!\bm{v}
        + \bm{r}
        - (1-\mu)\dfrac{\bm{r}+\mu\bm{\hat{x}}}{r_1^{3}}
        - \mu\dfrac{\bm{r}-(1-\mu)\bm{\hat{x}}}{r_2^{3}}
    \end{bmatrix},
\end{equation}
where $\mu \in \mathbb{R}_+$ is the mass ratio of the secondary body to the total system mass. The quantities $r_{1} = \|\bm{r} + \mu\bm{\hat{x}}\|$ and $r_{2} = \|\bm{r} - (1-\mu)\bm{\hat{x}}\|$ denote the distances from the spacecraft to the primary and secondary bodies, respectively. In the rotating synodic frame, the primary is located at $(-\mu,0,0)$ and the secondary at $(1-\mu,0,0)$. Consider a spacecraft in a slightly displaced position with respect to its nominal libration point orbit. The objective is to perform an impulsive maneuver that maintains the spacecraft in the vicinity of the nominal trajectory, while minimizing a weighted sum between the maneuver cost and the position errors at a number of downstream target points. The cost function to minimize is thus written as:
\begin{equation}
    J = \Delta\bm{v}_{t_c}^T\bm{Q}\Delta\bm{v}_{t_c}^T + \sum_{i=1}^{N_\textrm{t}} \bm{p}_{t_i}^T\bm{R}_i\bm{p}_{t_i},
    \label{eq:tpa_cost_function}
\end{equation}
where $\Delta\bm{v}_{t_c}$ is the impulsive maneuver performed at the control time $t_c$, $\bm{p}_{t_i}$ are the position errors with respect to the nominal trajectory at the ${N_\textrm{t}}$ target times $t_i$, and $\bm{Q},\bm{R}_i$ are $3\times 3$ diagonal matrices that weight the two competing terms. The initial state of the spacecraft is taken at $t_0<t_c$, and the target times are ordered chronologically such that $t_c<t_1<\ldots<t_{{N_\textrm{t}}}$.

Under the assumption that the dynamics can be linearized around the unperturbed trajectory, \cref{eq:tpa_cost_function} has the following analytical solution~\cite{Howell1993}. First, the position errors $\bm{p}_{t_i}$ are expressed as a linear combination of the initial perturbation and of the impulsive maneuver, i.e.:
\begin{equation}
    \bm{p}_{t_i} = \bm{B}_{t_i,t_c}\Delta\bm{v}_{t_c} + \bm{B}_{t_i,t_0}\bm{e}_{t_0} + \bm{A}_{t_i,t_0}\bm{p}_{t_0},
\end{equation}
where $\bm{p}_{t_0},\bm{e}_{t_0}$ are the initial deviations in position and velocity, respectively, and $\bm{A}_{t_i,t_j},\bm{B}_{t_i,t_j}$ are the $3\times 3$ submatrices of the \gls{stm} propagated along the nominal trajectory in the time interval $[t_j,t_i]$, where $j=\{0,c\}$:
\begin{equation}
    \bm{\Phi}_{t_i,t_j} = 
    \begin{bmatrix}
        \bm{A}_{t_i,t_j} & \bm{B}_{t_i,t_j} \\
        \bm{C}_{t_i,t_j} & \bm{D}_{t_i,t_j}
    \end{bmatrix}.
\end{equation}
The optimal maneuver is then obtained by setting the Jacobian of \cref{eq:tpa_cost_function} with respect to $\Delta\bm{v}_{t_c}$ equal to zero, and solving the resulting vector equality. The final expression for the optimal maneuver is:
\begin{equation}
    \label{eq:Journal/figs/tpa_opt_linear}
    \Delta\bm{v}_{t_c} = -\left(\bm{Q} + \sum_{i=1}^{N_\textrm{t}} \bm{B}_{t_i,t_c}^T\bm{R}_i\bm{B}_{t_i,t_c}\right)^{-1}
    \left[\sum_{i=1}^{N_\textrm{t}} \bm{B}_{t_i,t_c}^T\bm{R}_i\left(\bm{B}_{t_i,t_0}\bm{e}_{t_0} + \bm{A}_{t_i,t_0}\bm{p}_{t_0}\right)\right].
\end{equation}

Although \cref{eq:Journal/figs/tpa_opt_linear} provides a closed-form solution to the \gls{tpa} problem, the hypotheses introduced in its derivation might lead to large position errors at the downstream target points. This is especially true for large initial deviations or long propagation time spans, which cause the true trajectory to depart from the vicinity of the nominal orbit where the linear approximation holds. To maintain the spacecraft closer to its desired path, it is of interest to relax this assumption and compute a \gls{sk} maneuver that accounts for the nonlinearities in the underlying dynamics. As for the Keplerian dynamics example, the most accurate approach would be to formulate the \gls{tpa} problem as a \gls{nlp} problem, thus accounting for the true relative dynamics in its solution. This formulation, however, could quickly become computationally expensive, as every iteration of the \gls{nlp} solver requires the trajectory to be numerically propagated together with its partial derivatives with respect to the optimal maneuver needed to compute the update. Alternative approaches are thus desirable to reduce the computational burden while still providing an accurate solution.

Building on these considerations, two approaches to the solution to the \gls{tpa} problem are analyzed in this paper. These methods are based on higher order Taylor expansions of both the dynamics and the objective function in~\cref{eq:tpa_cost_function}. The latter are efficiently computed using \gls{da} techniques as described hereafter. The initial deviations from the nominal trajectory and the impulsive maneuver are firstly initialized similarly to \cref{eq:delta_v_da} as:
\begin{subequations}
    \label{eq:tpa_da_variables}
    \begin{align}
        \da{\bm{x}(t_0)} - \overline{\bm{x}}(t_0) &= \vdda{x}_0 \label{eq:tpa_da_dx0} \\
        \da{\Delta\bm{v}_{t_c}} &= \vdda{v}, \label{eq:tpa_da_dv}
    \end{align}
\end{subequations}
where $\overline{\bm{x}}(t_0)$ is the nominal initial state at $t_0$, $\da{\bm{x}(t_0)}$ the perturbed state, and $\da{\Delta\bm{v}_{t_c}}$ the impulsive maneuver at $t_c$. \Cref{eq:tpa_da_dx0} represents an infinitesimal deviation in the initial position and velocity of the spacecraft, while \cref{eq:tpa_da_dv} models an infinitesimal maneuver at $t_c$. The $k$-th order Taylor expansion of \cref{eq:tpa_cost_function} is then computed by taking into account \cref{eq:tpa_da_variables} for the propagation of the nominal trajectory in the \gls{da} framework, thus obtaining:
\begin{equation}
    \label{eq:tpa_poly_cost}
    \da{J} = \ty{k}{J}{\vdda{x}_0,\vdda{v}}.
\end{equation}
The first method~\cite{Fu2020} leverages Taylor map inversion techniques to compute a polynomial relationship between the initial perturbation $\vdda{x}_0$ and the optimal maneuver $\Delta\bm{v}_{t_c}$. This map is expressed as~\cite{Fu2020}:
\begin{equation}
    \label{eq:tpa_dv_map}
    \da{\Delta\bm{v}_{t_c}} = \ty{k}{\Delta\bm{v}_{t_c}}{\vdda{x}_0},
\end{equation}
and can be subsequently evaluated for several realizations of the initial deviations $\vdda{x}_0$ within the domain of validity of the map itself.

The second technique is instead based on the moment-\gls{sos} hierarchy introduced in \cref{s:mom-SOS}. Given an initial perturbation $\Delta\bm{x}_0$, \cref{eq:tpa_poly_cost} is firstly evaluated in $\Delta\bm{x}_0$ to obtain:
\begin{equation}
    \label{eq:tpa_poly_cost_given_dx}
    \left.\ty{k}{J}{\vdda{x}_0,\vdda{v}}\right|_{\vdda{x}_0=\Delta\bm{x}_0} = \ty{k}{J}{\vdda{v}},
\end{equation}
which is a polynomial expansion of the cost that depends only on the infinitesimal maneuver $\vdda{v}$. A \gls{pop} is then setup as:
\begin{equation}
    \label{eq:tpa_pop_obj}
    \min_{\vdda{v}} \ty{k}{J}{\vdda{v}}
\end{equation}
subject to:
\begin{equation}
    \label{eq:tpa_pop_constraint}
    \norm{\vdda{v}}_2^2 \leq \Delta v_{max}^2,
\end{equation}
where \cref{eq:tpa_pop_constraint} constraints the magnitude of the maneuver to values not greater than $\Delta v_{max}$. The optimization problem defined by \cref{eq:tpa_pop_obj,eq:tpa_pop_constraint} is then reformulated as a \gls{mop} and solved using the moment-\gls{sos} technique.

The numerical analyses that follow reproduce the results presented in Fu et al~\cite{Fu2020} for the Earth-Moon \gls{cr3bp} system. The mass parameter $\mu$, the characteristic length $\mathrm{LU}$, and the characteristic time $\mathrm{TU}$ for this system are listed in \cref{tab:cr3bp_quantities}.

\begin{table}[!ht]
    \caption{Mass parameter and characteristic quantities for the Earth-Moon \gls{cr3bp} system.}
    \label{tab:cr3bp_quantities}
    \centering
    \begin{tabular}{ccl}
        \toprule
        parameter & value & description \\ \midrule
        $\mu$ & \num{0.012150584395829193} & mass parameter \\
        $\mathrm{LU}$ & \SI{384400.0}{\km} & characteristic length \\
        $\mathrm{TU}$ & \SI{375190.2618944357}{\s} & characteristic time \\ \bottomrule
    \end{tabular}
\end{table}

The nominal orbit is a member of the $L_2$ northern Halo family with period equal to \num{5.9383} days. The unperturbed initial state at $t_0$ coincides with the periapsis of the Halo orbit. In normalized units this is equal to:
\begin{equation}
    \label{eq:x0_nominal_val}
    \overline{\bm{x}}(t_0) = 
    \begin{bmatrix}
        0.9876247155183294 \\
        0.0 \\
        -0.004607818604397411 \\
        0.0 \\
        2.271672949128728 \\
        0.0
    \end{bmatrix}.
\end{equation}
The impulsive maneuver is performed at $t_c=2\,\mathrm{days}$, and two control points are placed at $t_1=4.5996\,\mathrm{days}$ and $t_2=28.7710\,\mathrm{days}$, respectively. The perturbations on the initial orbit state are instead sampled in the eigenspace of $\bm{J}_{\Delta\bm{v}}^T\bm{J}_{\Delta\bm{v}}$, where $\bm{J}_{\Delta\bm{v}}$ is the Jacobian of the optimal maneuver $\Delta\bm{v}_{t_c}$ evaluated at $\Delta\bm{x}_0=\bm{0}$. This work uses a $100\times 100$ grid with $\bm{\nu}_i/\sqrt{\lambda_i}$ uniformly sampled in the interval $\pm\num{5e-3}$, where $i=\{1,2\}$, $\lambda_i$ are the two largest eigenvalues of $\bm{J}_{\Delta\bm{v}}^T\bm{J}_{\Delta\bm{v}}$, and $\bm{\nu}_i$ are the corresponding eigenvectors~\cite{Fu2020}. Moreover, all results are obtained with a fourth order Taylor expansion of the dynamics and of the constraints.

\Cref{fig:tpa_dv_opt} shows the optimal maneuver magnitude evaluated on this grid. \Cref{subfig:tpa_dv_opt_lin} correspond to the linearized approach, while \cref{subfig:tpa_dv_opt_pop} correspond to the solution to the \gls{pop}. The results obtained with \gls{da} map inversion are visually indistinguishable from the latter, and thus not displayed.

\begin{figure}[!ht]
    \centering
    \subcaptionbox{Linear solution.\label{subfig:tpa_dv_opt_lin}}{\includegraphics[clip, trim=1.25cm 0.25cm 0.25cm 0cm,width=0.5\textwidth]{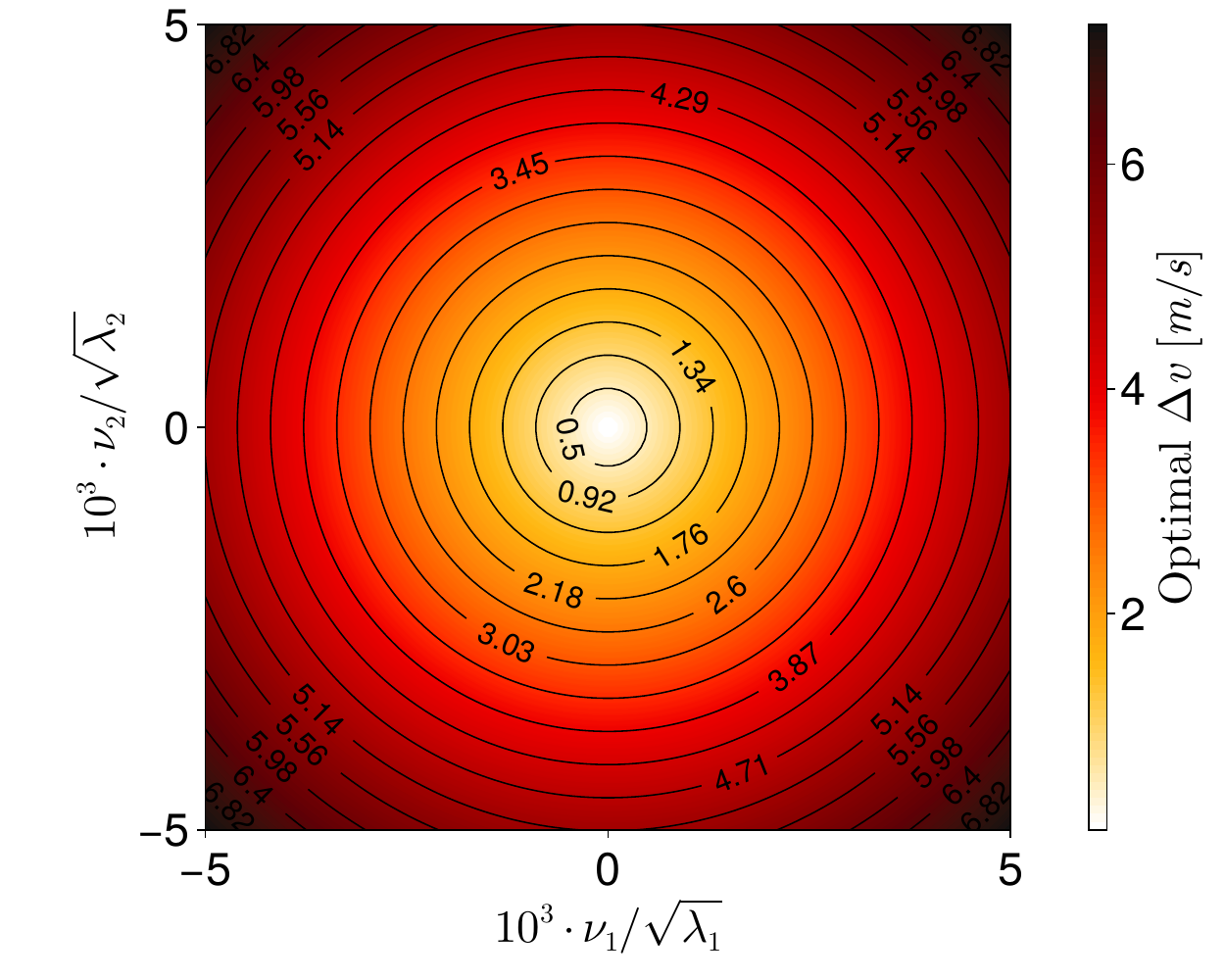}}%
    \hfill
    \subcaptionbox{\Gls{pop} solution.\label{subfig:tpa_dv_opt_pop}}{\includegraphics[clip, trim=1.25cm 0.25cm 0.25cm 0cm,width=0.5\textwidth]{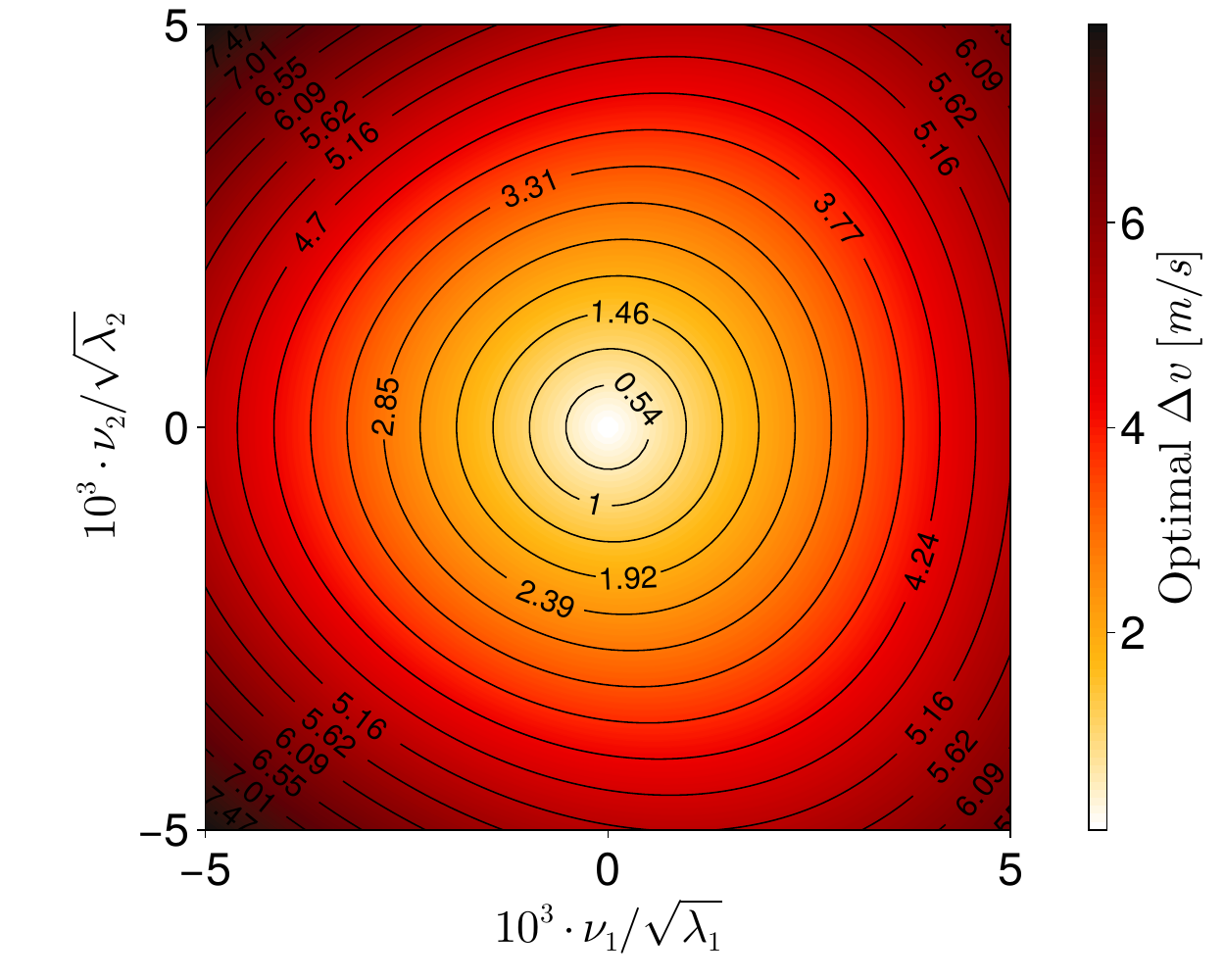}}
    \caption{Optimal maneuvers solutions to the \gls{tpa} problem.}
    \label{fig:tpa_dv_opt}
\end{figure}

\begin{figure}[!ht]
    \centering
    \subcaptionbox{Linear solution.\label{subfig:tpa_dv_err_lin}}{\includegraphics[clip, trim=0.5cm 0.25cm 0.25cm 0cm,width=0.5\textwidth]{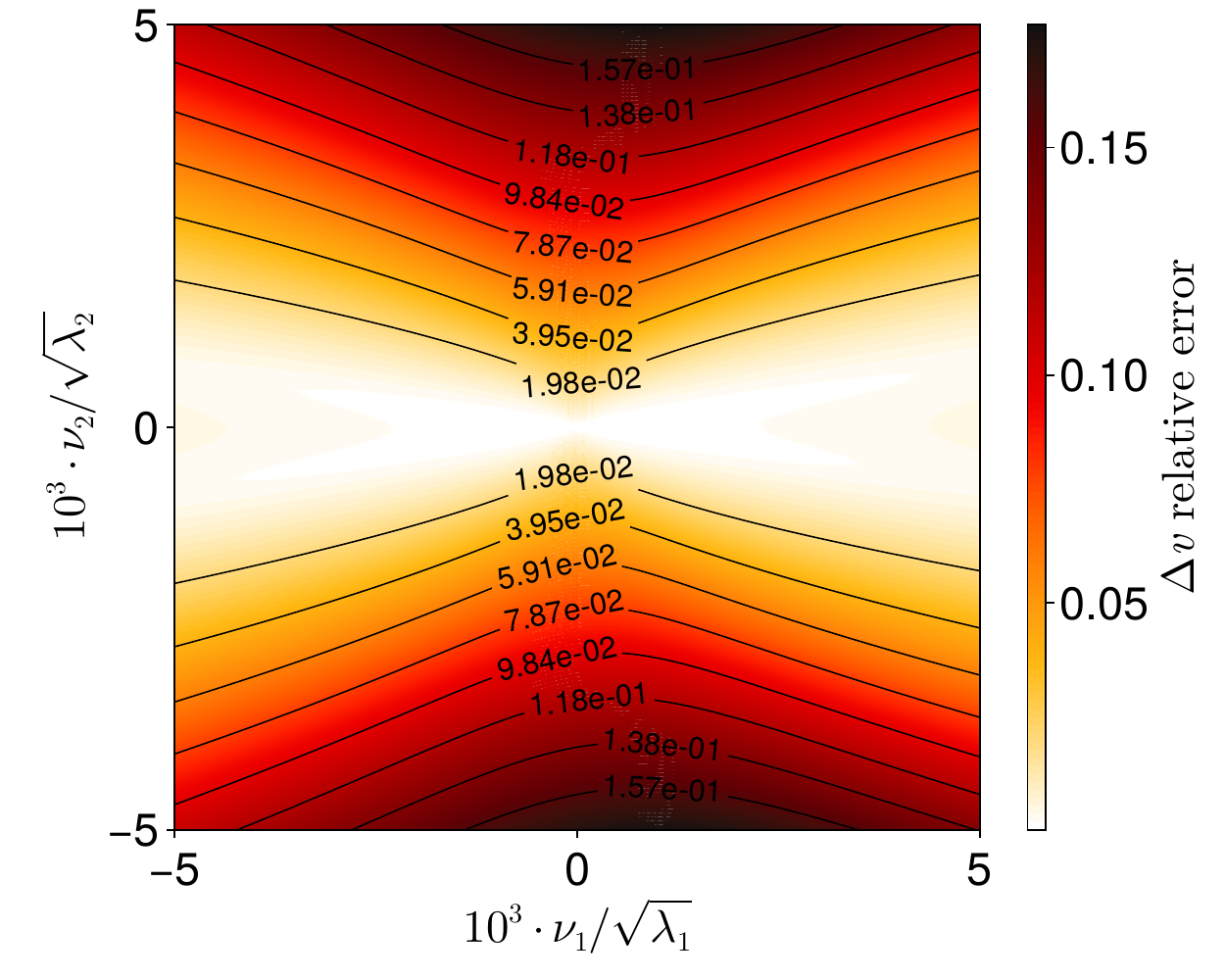}}%
    \hfill
    \subcaptionbox{\Gls{pop} solution.\label{subfig:tpa_dv_err_pop}}{\includegraphics[clip, trim=0.5cm 0.25cm 0.25cm 0cm,width=0.5\textwidth]{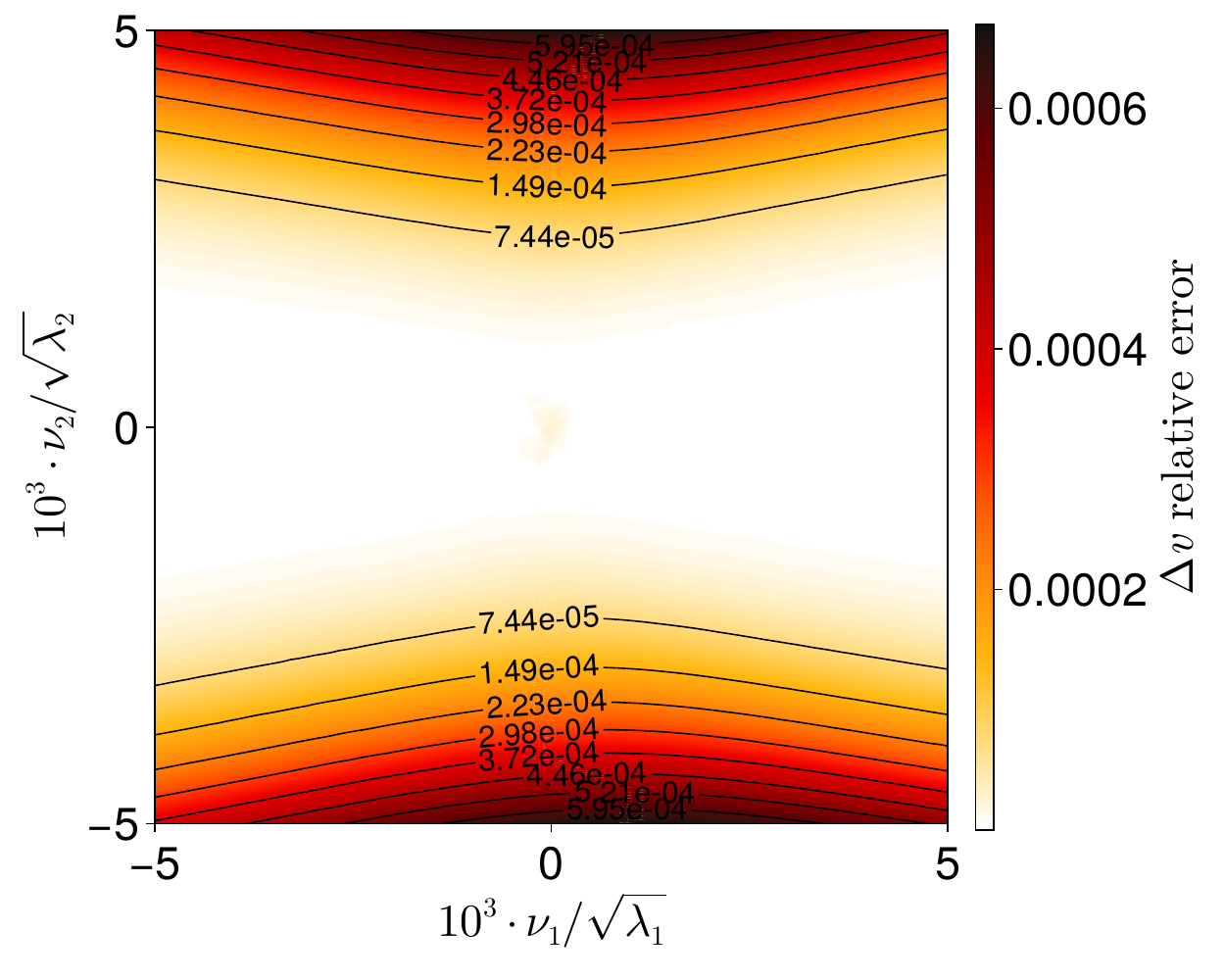}}
    \caption{Relative errors in the optimal maneuvers compared to the \gls{nlp} solution to the \gls{tpa} problem.}
    \label{fig:tpa_dv_err}
\end{figure}

\Cref{subfig:tpa_dv_opt_lin} clearly shows the limitations of the linearized technique, as contour lines are concentric circumferences centered on the nominal solution. In contrast, \cref{subfig:tpa_dv_opt_pop} depicts a more complex structure which closely matches the pointwise solution to the \gls{nlp} formulation. This is demonstrated in \cref{fig:tpa_dv_err}, which shows the relative errors in the optimal maneuver magnitude with respect to the solution to the \gls{nlp} problem. Comparing \cref{subfig:tpa_dv_err_lin} with \cref{subfig:tpa_dv_err_pop}, it is in fact seen that the polynomial approach is about three orders of magnitude more accurate than the linearized method.

\begin{table}[!ht]
    \caption{Errors in the computed optimal maneuvers for the \gls{tpa} problem. Mean and maximum relative difference between \gls{nlp} solution and other optimization techniques.}
    \newcommand{\scinum}[1]{\tablenum[exponent-mode=scientific, round-precision=3, round-mode=places, table-format=1.3e-1]{#1}}
    \begin{filecontents*}{Journal/data/errors_statistics_4D9_100x100.csv}
method,dverrmean,dverrmax,drerrmeanti,drerrmaxti,drerrmeantii,drerrmaxtii
linear,0.06693578925543724,0.17696385370210607,0.00022146040054839123,0.0006669238711559908,0.0012508986901976814,0.004462692187802194
map inversion,0.0001374618603508769,0.0006716146168120227,4.6934578501911675e-7,2.368028177535581e-6,2.7713287713068435e-6,1.559009867153447e-5
GMP (primal),0.0001382411538710121,0.0006697556485130614,4.698237748701755e-7,2.366185899137746e-6,2.7750019429272895e-6,1.5504709481683417e-5
SOS (dual),0.00013730243494308256,0.0006696587497003019,4.6847836734922684e-7,2.361324565245795e-6,2.7707785898308935e-6,1.5463154892055227e-5
\end{filecontents*}
\csvreader[
        head to column names,
        before reading = \begin{center},
        tabular = rcc,
        table head = \toprule method & {$\epsilon_{\Delta v,\mathrm{mean}}$} & {$\epsilon_{\Delta v,\mathrm{max}}$} \\\midrule,
        table foot = \bottomrule,
        after reading = \end{center},
    ]{Journal/data/errors_statistics_4D9_100x100.csv}{}{%
        \method &
        \scinum{\dverrmean} &
        \scinum{\dverrmax}
    }
    \label{tab:tpa_maneuvers_errors}
\end{table}

The mean and maximum relative errors extracted from \cref{fig:tpa_dv_err} are reported in \cref{tab:tpa_maneuvers_errors}. The last three rows correspond to the solution obtained with \gls{da} map inversion, and to those obtained by solving the \gls{pop} defined by \cref{eq:tpa_pop_obj,eq:tpa_pop_constraint}, respectively. \Gls{gmp} refers to the solution to the \gls{lp} problem in \cref{eq:lp}, while \gls{sos} indicates that the dual to this problem is solved instead. As expected, these three solutions are very close to each others, as they are based on the same fourth order Taylor expansion of the problem. In contrast, the linearized solution is on average two orders of magnitude less accurate, and up to three orders of magnitude worse across the selected sampling space.

\begin{figure}[!ht]
    \centering
    \subcaptionbox{Linear solution at $t_1$.\label{subfig:tpa_pos_err_lin_t1}}{\includegraphics[clip, trim=0.5cm 0.25cm 0cm 0cm,width=0.5\textwidth]{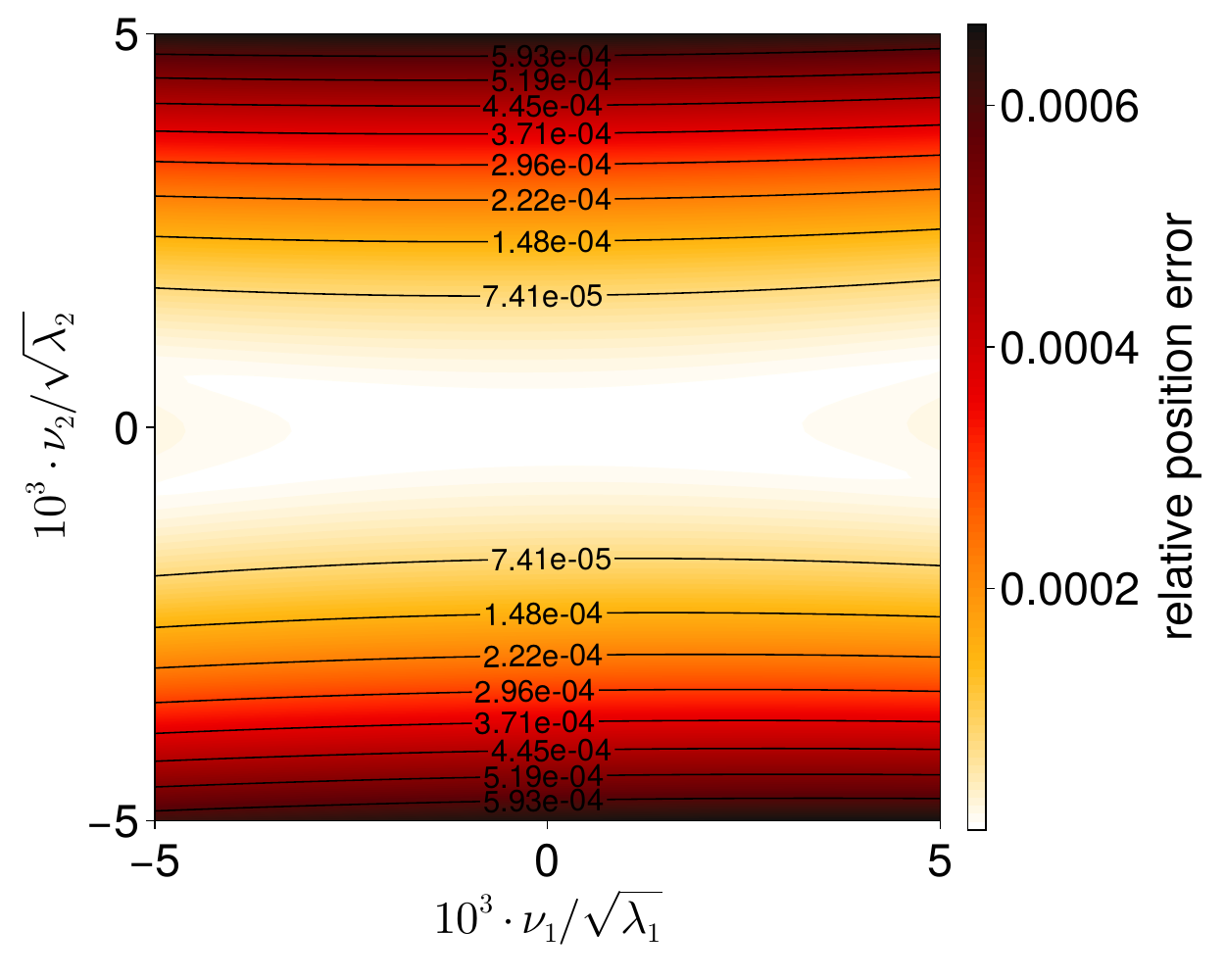}}%
    \hfill
    \subcaptionbox{\Gls{pop} solution at $t_1$.\label{subfig:tpa_pos_err_pop_t1}}{\includegraphics[clip, trim=0.25cm 0.25cm 0.25cm 0cm,width=0.5\textwidth]{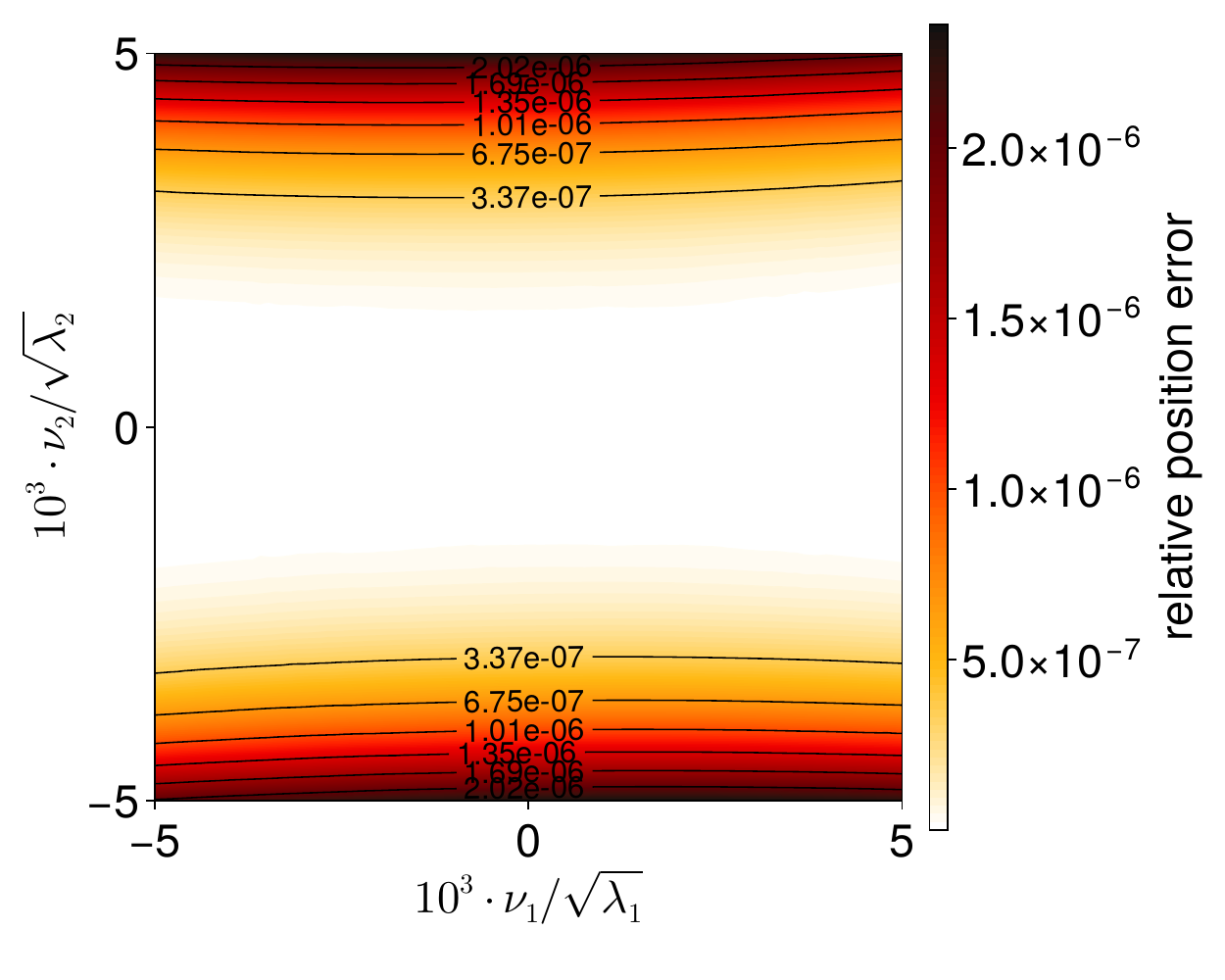}}
    \caption{Relative error in position at the first target point compared to the \gls{nlp} solution to the \gls{tpa} problem.}
    \label{fig:tpa_pos_err_t1}
\end{figure}

\begin{figure}[!ht]
    \centering
    \subcaptionbox{Linear solution at $t_2$.\label{subfig:tpa_pos_err_lin_t2}}{\includegraphics[clip, trim=0.5cm 0.25cm 0cm 0cm,width=0.5\textwidth]{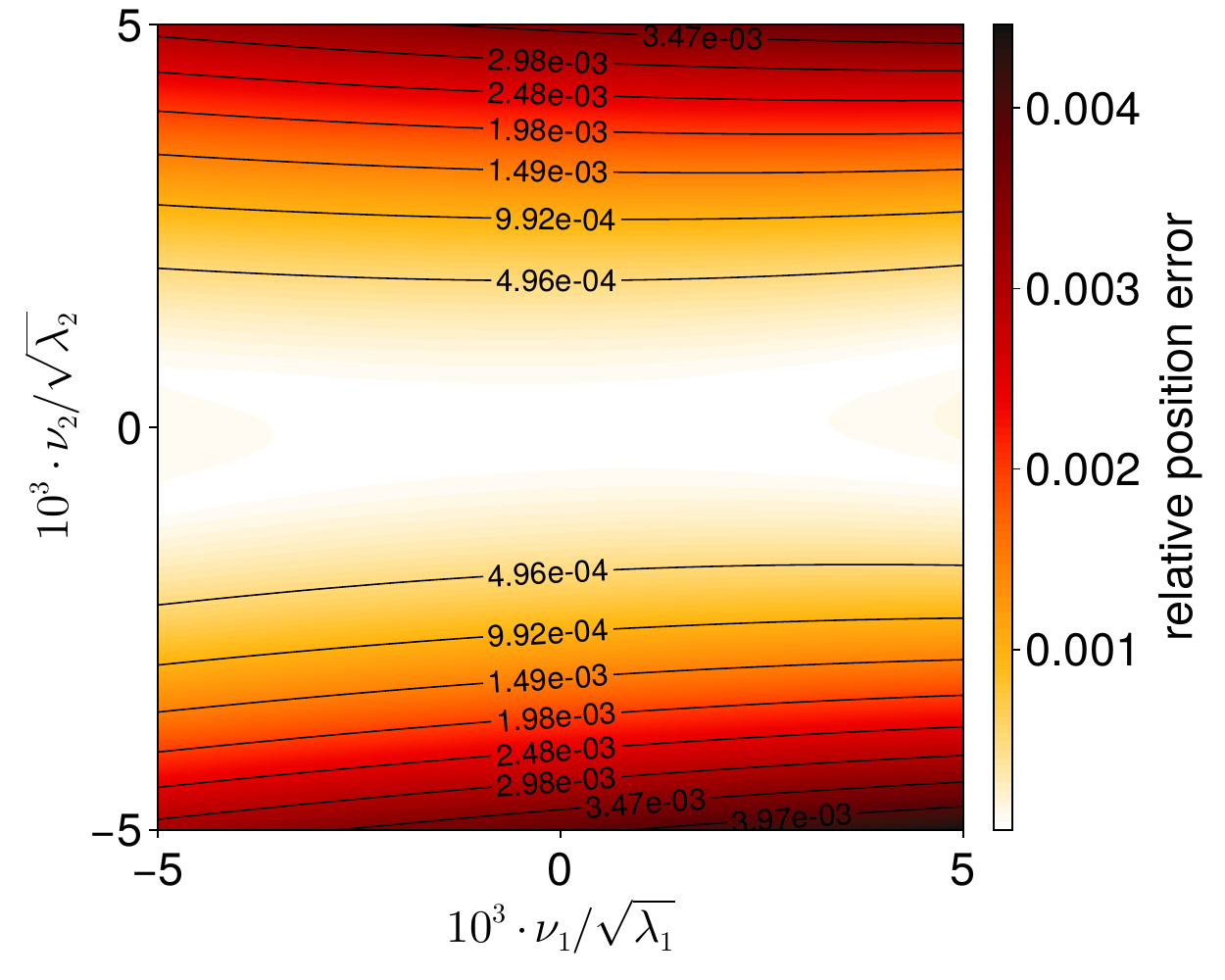}}%
    \hfill
    \subcaptionbox{\Gls{pop} solution at $t_2$.\label{subfig:tpa_pos_err_pop_t2}}{\includegraphics[clip, trim=0.25cm 0.25cm 0.25cm 0cm,width=0.5\textwidth]{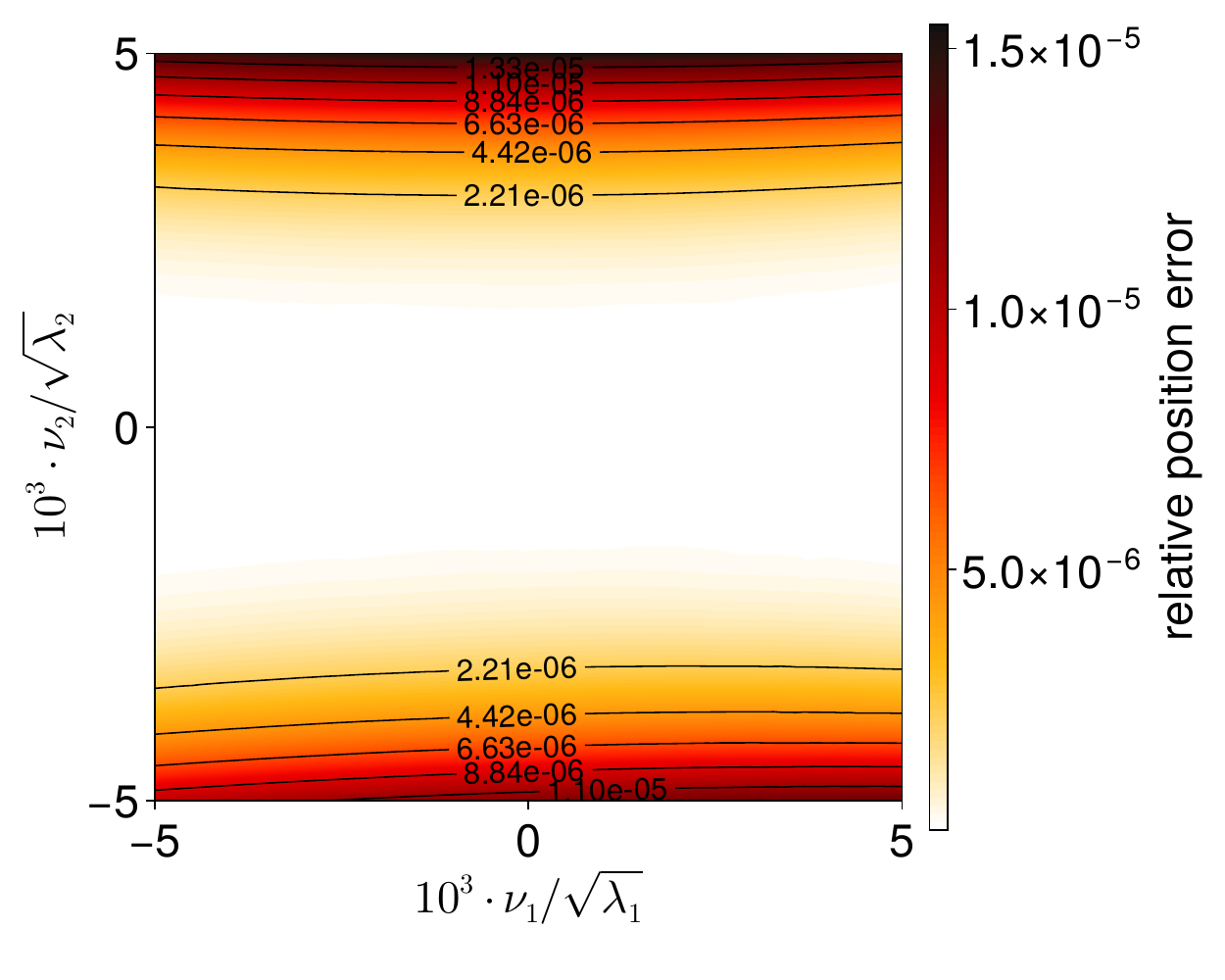}}
    \caption{Relative error in position at the second target point compared to the \gls{nlp} solution to the \gls{tpa} problem.}
    \label{fig:tpa_pos_err_t2}
\end{figure}

Similarly to \cref{fig:tpa_dv_err}, \cref{fig:tpa_pos_err_t1,fig:tpa_pos_err_t2} shows the relative errors in position at the two target points $t_1$ and $t_2$, respectively. \Cref{subfig:tpa_pos_err_lin_t1,subfig:tpa_pos_err_lin_t2} correspond to the errors between the linearized and \gls{nlp} solutions, while \cref{subfig:tpa_pos_err_pop_t1,subfig:tpa_pos_err_pop_t2} compare the nonlinear solution with that to the \gls{pop}. The errors obtained with \gls{da} map inversion are also in this case indistinguishable from those in \cref{subfig:tpa_pos_err_pop_t1,subfig:tpa_pos_err_pop_t2}.

These four plots confirm that the linearized approach can be up to three orders of magnitude less accurate than the solutions obtained with a higher order Taylor expansion of the dynamics and of the cost function. This fact is highlighted in \cref{tab:tpa_position_errors}, which reports a trend similar to that in \cref{tab:tpa_maneuvers_errors}. Indeed, also in this case the linearized solution is about two to three orders of magnitude less accurate than its polynomial based counterparts.

\begin{table}[!ht]
    \caption{Position errors at target points for the \gls{tpa} problem. Mean and maximum relative difference between \gls{nlp} solution and other optimization techniques.}
    \newcommand{\scinum}[1]{\tablenum[exponent-mode=scientific, round-precision=3, round-mode=places, table-format=1.3e-1]{#1}}
    \csvreader[
        head to column names,
        before reading = \begin{center},
        tabular = rcccc,
        table head = \toprule method & {$\epsilon_{\Delta p_1,\mathrm{mean}}$} & {$\epsilon_{\Delta p_1,\mathrm{max}}$} & {$\epsilon_{\Delta p_2,\mathrm{mean}}$} & {$\epsilon_{\Delta p_2,\mathrm{max}}$}\\\midrule,
        table foot = \bottomrule,
        after reading = \end{center},
    ]{Journal/data/errors_statistics_4D9_100x100.csv}{}{%
        \method &
        \scinum{\drerrmeanti} &
        \scinum{\drerrmaxti} &
        \scinum{\drerrmeantii} &
        \scinum{\drerrmaxtii}
    }
    \label{tab:tpa_position_errors}
\end{table}

Performance metrics for the \gls{tpa} problem are finally reported in \cref{tab:tpa_iterations,tab:tpa_runtime}. \Cref{tab:tpa_iterations} summarizes the number of iterations needed for the iterative optimization methods to converge. These are the solution to the \gls{nlp} problem using \gls{ipopt}, and the solution to the \gls{sdp} problems that arise from the relaxation of the \glspl{pop}. The three values $i_{\mathrm{min}}$, $i_{\mathrm{mean}}$, and $i_{\mathrm{max}}$ correspond to the minimum, mean, and maximum number of iterations over the $100\times 100$ sampling grid, respectively. Solving the \gls{nlp} always require the smallest number of iterations among the three methods. The minimum and mean values are in fact approximately one fourth than those obtained with the second best approach. Solutions to the primal and dual \gls{mop} require instead a similar number of iterations, with the \gls{gmp} solution taking the edge in the minimum and mean values. In contrast, the maximum number of iterations is smaller for the solution to the dual problem.

\begin{table}[!ht]
    \caption{Minimum, mean, and maximum number of iterations for solving the \gls{tpa} problem.}
    \newcommand{\scinum}[1]{\tablenum[round-precision=3, round-mode=places, table-format=2.3]{#1}}
    \begin{filecontents*}{Journal/data/iterations_statistics_4D9_100x100.csv}
method,itrmin,itrmean,itrmax
nonlinear,2,3.1032,16
GMP (primal),8,12.2105,27
SOS (dual),10,15.152,19
\end{filecontents*}
\csvreader[
        head to column names,
        before reading = \begin{center},
        tabular = rccc,
        table head = \toprule method & {$i_{\mathrm{min}}$} & {$i_{\mathrm{mean}}$} & {$i_{\mathrm{max}}$} \\\midrule,
        table foot = \bottomrule,
        after reading = \end{center},
    ]{Journal/data/iterations_statistics_4D9_100x100.csv}{}{%
        \method &
        \itrmin &
        \scinum{\itrmean} &
        \itrmax
    }
    \label{tab:tpa_iterations}
\end{table}

\Cref{tab:tpa_runtime} reports instead the mean and median runtimes for all optimization methods presented in this paper. In contrast to \cref{tab:tpa_iterations}, these values were not averaged over the sampling grid. Instead, the solution for a specific realization of the initial perturbation was benchmarked multiple times using the dedicated Julia package \texttt{BenchmarkTools.jl}\footnote{\url{https://github.com/JuliaCI/BenchmarkTools.jl}}.

\begin{table}[!ht]
    \caption{Mean and median runtimes for solving the \gls{tpa} problem.}
    \newcommand{\scinum}[1]{\tablenum[round-precision=3, round-mode=places, table-format=4.3]{#1}}
    \newcommand{\expnum}[1]{\tablenum[round-precision=0, round-mode=places, table-format=1e1, exponent-mode=scientific]{#1}}
    \begin{filecontents*}{Journal/data/runtime_benchmark.csv}
method,timemean,timemedian,calls
linear solution,1.2982084648070908,1.301083,10000
nonlinear solution,21.48291217167382,21.499916,10000
polynomial map computation,2684.5624585,2684.5624585,1
polynomial map inversion,0.3672093614,0.360334,1
polynomial map evaluation,0.0309467587,0.031584,10000
NLP solution of POP,0.4442180371,0.416667,10000
GMP (primal) solution of POP,1.9248091328967643,1.8195625,10000
SOS (dual) solution of POP,1.6467723414634148,1.630083,10000
\end{filecontents*}
\csvreader[
        head to column names,
        before reading = \begin{center},
        tabular = rccc,
        table head = \toprule method & {$t_{\mathrm{mean}}$, \unit{\milli\second}} & {$t_{\mathrm{median}}$, \unit{\milli\second}} & calls on $100\times 100$ grid\\\midrule,
        table foot = \bottomrule,
        after reading = \end{center},
    ]{Journal/data/runtime_benchmark.csv}{}{%
        \method &
        \scinum{\timemean} &
        \scinum{\timemedian} &
        \expnum{\calls}
    }
    \label{tab:tpa_runtime}
\end{table}

As expected, the linearized approach is the most efficient when a single solution is sought. The solution to the \gls{nlp} problem is about twenty times slower than the former, while the polynomial based methods are heavily penalized by the upfront cost required to compute the higher order Taylor map. However, if multiple solutions need to be computed around the same nominal trajectory, the \gls{da} map inversion technique might become the most efficient optimization method. As reported on the last column of \cref{tab:tpa_runtime}, the polynomial map is in fact computed and inverted only once. Each realization of the initial perturbation then requires a single and fast evaluation of the precomputed inverse map. This would be the case in a \gls{mc} simulation where several realizations of the initial perturbation are drawn from a given multivariate probability distribution. The solution to the \gls{pop} is instead more expensive, with both the primal and dual approaches taking slightly longer than the linearized solution. In addition to these recurring costs, the \gls{pop} based methods rely on the same high order Taylor expansion as the map inversion technique, which adds a considerable initial cost to the overall runtime. A possible mitigation is to treat the \gls{pop} as a generic \gls{nlp}, and solve this problem with a general purpose optimizer such as \gls{ipopt}. If on one side the runtime is cut by a factor of about four, on the other hand all guarantees of a robust convergence to the global optimum are lost. The same argument also applies to the \gls{da} map inversion technique, which can be seen as a local optimizer. Therefore, if the cost function is non-convex and characterized by multiple local optima, the chances of converging to the global minimizer are strongly influenced by either the initial guess provided to the \gls{nlp} solver, or by the expansion point chosen for the high order Taylor map. In contrast, the solution to the \gls{pop} is always guaranteed to be the global optimum (or optima) within the bounded set in which the optimization variables are defined.

% ###############################################################################################
% ###############################################################################################

\section{Low-Thrust Station-Keeping in the CR3BP}\label{s:LTTargetingProblem}
Having demonstrated the efficacy of the moment-\gls{sos} hierarchy for the impulsive targeting problem, this section extends the proposed framework to a low-thrust \gls{sk} problem in the \gls{cr3bp}. The incorporation of a varying time thrust segment and multiple targeting points increases the dimensionality of the decision space, requiring the optimizer to navigate the trajectory's sensitivity to the applied low-thrust profile.

Such a problem can be formulated as follows. The low-thrust version of the \gls{cr3bp} dynamics introduced in \cref{eq:cr3bp_ballistic_dynamics} can be written as:

\begin{equation}\label{eq:cr3bp_dynamics_LT}
    \bm{f}(\bm{x},\bm{\alpha},u) =
    \begin{bmatrix}
        \bm{v} \\[6pt]
        2\,\bm{\hat{z}}\!\times\!\bm{v}
        + \bm{r}
        - (1-\mu)\dfrac{\bm{r}+\mu\bm{\hat{x}}}{r_1^{3}}
        - \mu\dfrac{\bm{r}-(1-\mu)\bm{\hat{x}}}{r_2^{3}}
        + \dfrac{T_{\textrm{max}}}{m}\bm{\alpha} \\[10pt]
        -\dfrac{T_{\textrm{max}}}{I_\textrm{sp}g_0}\,u
    \end{bmatrix},
\end{equation}
\\
where $\bm{x}\in\mathbb{R}^{7\times 1}$ is the state vector comprising the spacecraft position $\bm{r} \in \mathbb{R}^{3\times 1}$, velocity $\bm{v} \in \mathbb{R}^{3\times 1}$, and mass $m \in \mathbb{R}_{+}$; $u \in [0,1]$ is the control magnitude; $\bm{\alpha} \in \mathbb{R}^{3\times 1}$ is the control vector in the reference frame, which encodes both the control direction and magnitude, thus $\|\bm{\alpha}\|=u$, a parametrization that decouples the magnitude from the control vector to assist in numerical stability; $T_{\textrm{max}} \in \mathbb{R}_+$ and $I_\textrm{sp} \in \mathbb{R}_+$ are the maximum thrust and specific impulse of the propulsion system, respectively; and $g_0 = 9.80665\ \textrm{m/s}^2$ is the standard gravitational acceleration of the Earth.

The nominal orbit is an Earth-Moon \gls{nrho} from Ref. \cite{Fu2020}. The spacecraft is assumed to start at periapsis, and perform a low-thrust maneuver over a time of $t_\textrm{thrust}$. The control vector is defined to be fixed in the \gls{rtn} reference frame of the spacecraft, and is therefore time-varying in the synodical \gls{cr3bp} frame. Due to the thrust being bounded, the control authority consequently limits the possible set of trajectory corrections. To ensure numerical feasibility of the problem, the reduction of the tracking error at the $N_\textrm{t}$ target points must be incorporated as a minimization objective within the cost function, as oppose to an inequality constraint as in the unbounded impulsive case of \cref{eq:mahalanobis_distance}. The objective function is thus stated as:

\begin{equation}\label{eq:objective_function_LT}
    \min_{\{\Delta \bm{\alpha}_i, \Delta u_i\}} \quad J = \sum_{i=1}^{N_\textrm{t}} \bm{p}_{t_i}^T\bm{R}_i\bm{p}_{t_i} + t_\textrm{thrust},
\end{equation}
where $\bm{p}_{t_i} \in \mathbb{R}^{3\times 1}$ are the position errors with respect to the nominal trajectory at the ${N_\textrm{t}}$ target times $t_{\textrm{target},i}$, and $\bm{R}_i$ is a $3\times 3$ diagonal matrix that weights the competing terms. In addition, the equality constraints:

\begin{subequations}\label{eq:equality_constraints_LT}
    \begin{align}
        \bm{x}(t) &= \bm{x}(t_0) + \int_{t_0}^{t} \bm{f}(\bm{x}(\tau),\bm{\alpha}(\tau),u(\tau),\tau)\dd{\tau}, \label{eq:state_propagation_LT} \\
        \bm{x}(t_0) &= \overline{\bm{x}}_0, \\
        u(t) &= \begin{cases} 
        1 & \text{for } t \in [t_0, t_{\text{thrust}}] \\ 
        0 & \text{otherwise} 
        \end{cases} \label{eq:boolean_control_constraint}
    \end{align}
\end{subequations}
and inequality constraints:

\begin{equation}\label{eq:control_vector_constraint}
        \|\overline{\bm{\alpha}}(t) + \Delta \bm{\alpha}(t)\|^2_2 \leq u(t)^2, \quad \forall \ t,
\end{equation}
apply, where $\overline{\bm{x}}_0 \in \mathbb{R}^{7\times1}$ is the initial state before applying control, and $\overline{\bm{\alpha}} \in \mathbb{R}^{3\times1}$ and $\Delta \bm{\alpha}  \in \mathbb{R}^{3\times1}$ are the control vector on the reference control profile and update to the control vector, respectively. \Cref{eq:control_vector_constraint} constrains the magnitude of the control vector to be no greater than the scalar control magnitude. At optimality, this inequality constraint is known to be binding. The problem is thus to choose the control $\Delta \bm{\alpha}$ to minimize the cost function \cref{eq:objective_function_LT} subject to the constraints of \cref{eq:equality_constraints_LT,eq:control_vector_constraint}.

\subsection{Polynomial Optimization Problem}
The \gls{nlp} problem formulated in \cref{eq:objective_function_LT,eq:equality_constraints_LT,eq:control_vector_constraint} is inherently difficult to solve, due to the trajectory's high sensitivity to control perturbations and the presence of non-convex constraints. The problem is therefore converted into a \gls{pop} to subsequently leverage the polynomial optimization techniques of the moment-\gls{sos} hierarchy. The control corrections are firstly initialized as:

\begin{equation}
        \left[ \bm{\alpha} \right] = \overline{\bm{\alpha}} + \delta \bm{\alpha},
    \label{eq:delta_LT_da}
\end{equation}
where $\delta \bm{\alpha}$ are the independent \gls{da} variables corresponding to the control vector. Then, evaluating \cref{eq:equality_constraints_LT} in the \gls{da} framework results in a high-order polynomial expression for the state propagation. The first state to be propagated is the low-thrust segment, starting from $\overline{\bm{x}}_0$. As this requires an expansion point, an initial guess of the thrust time is chosen to be $\overline{t}_\textrm{thrust}$. Thus, after propagating the state, the resulting expression is

\begin{equation}
    \da{\bm{x}(\overline{t}_\textrm{thrust})} = \ty{k}{\bm{x}(\overline{t}_\textrm{thrust})}{\delta \bm{\alpha}}.
    \label{eq:thrust_expansion}
\end{equation}

As the \gls{pop} is solving for a variable thrust segment, the high-order expansion of the flow with respect to the thrust time is required. The thrust segment is first initialized within the \gls{da} framework as

\begin{equation}
        \left[ t_\textrm{thrust} \right] = \overline{t}_\textrm{thrust} + \overline{t}_\textrm{thrust}\delta t_\textrm{thrust},
\end{equation}

Note that the \gls{da} variable $\delta t_\textrm{thrust}$ is multiplied by $\overline{t}_\textrm{thrust}$, which acts as a scaling factor. Constraining $\delta t_\textrm{thrust}\in[-1,1]$, when $\delta t_\textrm{thrust}=-1$, then $\left[ t_\textrm{thrust} \right]=0$, equating to zero thrust, and the opposite bound results in a thrust segment of $2t_{\textrm{thrust}}$. Although one may constrain $\delta t_\textrm{thrust}\in[-1,\infty)$ to exhaust all eventualities, this normalization $\delta t_\textrm{thrust}\in[-1,1]$ promotes the numerical stability of the problem, and is adopted for this work. A fixed-point iteration method known as the Picard iteration~\cite{PicardIteration}, which has previously been utilized effectively for \gls{da}-based time expansions~\cite{Evans2024_1}, is then applied. The iteration has the form

\begin{equation}\label{eq:PicardIteration}
    \begin{gathered}
        \bm{x}_{k+1} = \bm{x}_{\textrm{thrust}} + \int_{\overline{t}_{\textrm{thrust}}}^{t_{\textrm{thrust}}} \dot{\bm{x}}\left(\bm{x}_{k}(\tau)\right) \dd{\tau}
        \\
        \bm{x}_{k=0} = \bm{x}_{\textrm{thrust}}
    \end{gathered}
\end{equation}
where $\tau$ is a dummy variable for the time integration. The Picard–Lindelöf theorem states that, if the differential equations $\dot{\bm{x}}$ are Lipschitz continuous in $\bm{x}$, then there exists a unique solution to which the sequence of Picard iterations converges~\cite{PicardBanach}. This procedure is deterministic and guaranteed to converge. Thus, after $k$ iterations, the exact $k^{\text{th}}$-order Taylor expansion of the solution flow with respect to the final time is obtained, building upon the initial state previously determined from \cref{eq:thrust_expansion}. This temporal expansion allows variations in the time-of-flight to be considered directly. The resulting polynomial is given by

\begin{equation}
    \da{\bm{x}(t_\textrm{thrust})} = \ty{k}{\bm{x}(t_\textrm{thrust})}{\delta \bm{\alpha},\delta t_{\textrm{thrust}}}.
    \label{eq:thrust_expansion_dt}
\end{equation}

Subsequently, the trajectory is propagated from $t_\textrm{thrust}$ to the first target time $\overline{t}_1$. This yields a high-order expansion of the first target state that captures sensitivities to variations in both the control vector and thrust duration. The time at the first target is then initialized within the \gls{da} framework as 

\begin{equation}
        \left[ t_{\textrm{coast}} \right] = \overline{t}_{\textrm{coast}} + t_{\textrm{thrust}}\delta t_{\textrm{coast}},
        \label{eq:coast_DA_initialization}
\end{equation}
before the Picard iteration is again employed to perform a temporal expansion, this time accounting for variations in the coast duration between engine shutdown and target arrival. The resulting polynomial expansion is denoted as

\begin{equation}
    \da{\bm{x}(t_1)} = \ty{k}{\bm{x}(t_1)}{\delta \bm{\alpha},\delta t_{\textrm{thrust}},\delta t_1}.
    \label{eq:target_state_expansion_LT}
\end{equation}

Propagation until subsequent target times can then be performed as normal, without further time expansions. Note the scaling factor of $t_{\textrm{thrust}}$ multiplying the \gls{da} variable corresponding to the change in the coast segment in \cref{eq:coast_DA_initialization}. Importantly, the objective is to find a trajectory in which the extension of the thrust segment equals the reduction of the coast segment, ensuring the overall orbital timeline remains unchanged. Thus, this condition is satisfied with $\delta t_{\textrm{coast}}\in[-1,1]$, promoting numerical stability of the problem.

Alternatively, the condition of the change in the coast segment equaling the negative change in the thrust segment can be achieved via an inverse mapping procedure, akin to those employed in previous works~\cite{Evans2024_1,Evans2024_2}. However, because the polynomial in \cref{eq:target_state_expansion_LT} is of $n$-th order in both $\delta t_{\text{thrust}}$ and $\delta t_1$ (including all cross-terms), enforcing this condition at this stage generates higher-order terms up to order $2n$. Subsequent truncation to the polynomial's maximum order $n$ would then discard these newly created terms, resulting in a severe loss of information. To prevent this, the condition is enforced directly within the \gls{pop}, thereby preserving the full fidelity of the temporal expansions.

The tracking error at each target node $d_i(t_i) = \bm{p}_{t_i}^T\bm{R}_i\bm{p}_{t_i}$ is thus approximated via Taylor polynomials by substituting for the positional elements of \cref{eq:target_state_expansion_LT}. To promote numerical stability, this is normalized by the tracking error on the uncorrected trajectory profile. This conveniently corresponds to the zero-th order term of the expansion. The normalized Taylor approximation of the tracking error is denoted by

\begin{equation}
    [\tilde{d}_i] = \frac{\ty{k}{d_i}{\delta \bm{\alpha},\delta t_{\textrm{thrust}},\delta t_{\textrm{coast}}}}{\mathcal{T}^{(0)}_{d_i}}.
    \label{eq:tracking_error_expansion_LT}
\end{equation}

If one were to set the metric $\bm{R}_i=\mathrm{diag}(1,1,1)$ and consider zero control corrections, $[\tilde{d}_i]\big|_{\delta \bm{\alpha}=\bm{0}} = 1.0$. Consequently, any feasible solution will result in $[\tilde{d}_i] \in (0,1]$. This term would then be on a comparable scale to the competing fuel expenditure term in \cref{eq:objective_function_LT}, improving the numerical stability of the \gls{pop}. Substituting \cref{eq:tracking_error_expansion_LT} into \cref{eq:objective_function_LT} and additionally substituting the control and thrust time variables with their \gls{da} representations $\delta \bm{\alpha}$ and $[t_\textrm{thrust}]$ respectively, the cost function is approximated within the \gls{da} framework as:

\begin{equation}\label{eq:objective_function_DA_LT}
    \quad [J] = \ty{k}{J}{\delta \bm{\alpha},\delta t_{\textrm{thrust}},\delta t_{\textrm{coast}}} = \sum_{i=1}^{N_\textrm{t}} [\tilde{d}_i] + [t_{\textrm{thrust}}].
\end{equation}

The \gls{pop} for the low-thrust \gls{sk} scenario is then defined as:

\begin{equation}\label{eq:pop_problem_LT}
    \min_{\{\delta \bm{\alpha}, \delta t_{\textrm{thrust}}, \delta t_{\textrm{coast}} \}} \, \ty{k}{J}{\delta \bm{\alpha},\delta t_{\textrm{thrust}},\delta t_{\textrm{coast}}}
\end{equation}
subject to:
\begin{equation}\label{eq:pop_constraint_LT}
    \begin{aligned}
        (\overline{\bm{\alpha}} + \delta \bm{\alpha})^\top\cdot(\overline{\bm{\alpha}} + \delta \bm{\alpha}) &\leq 1 \\
        (\delta t_\textrm{thrust})^2 &\leq 1 \\
        (\delta t_\textrm{coast})^2 &\leq 1 \\
        \delta t_\textrm{thrust}-\delta t_{\textrm{coast}} &= 0
        \end{aligned}
\end{equation}

\subsection{Numerical Application}

\begin{figure}[!ht]
    \centering
    \hfill
    \subcaptionbox{Trajectory due to moment-\gls{sos} based \gls{sk}.\label{subfig:LT_corrected_orbits}}{\includegraphics[width=0.48\textwidth]{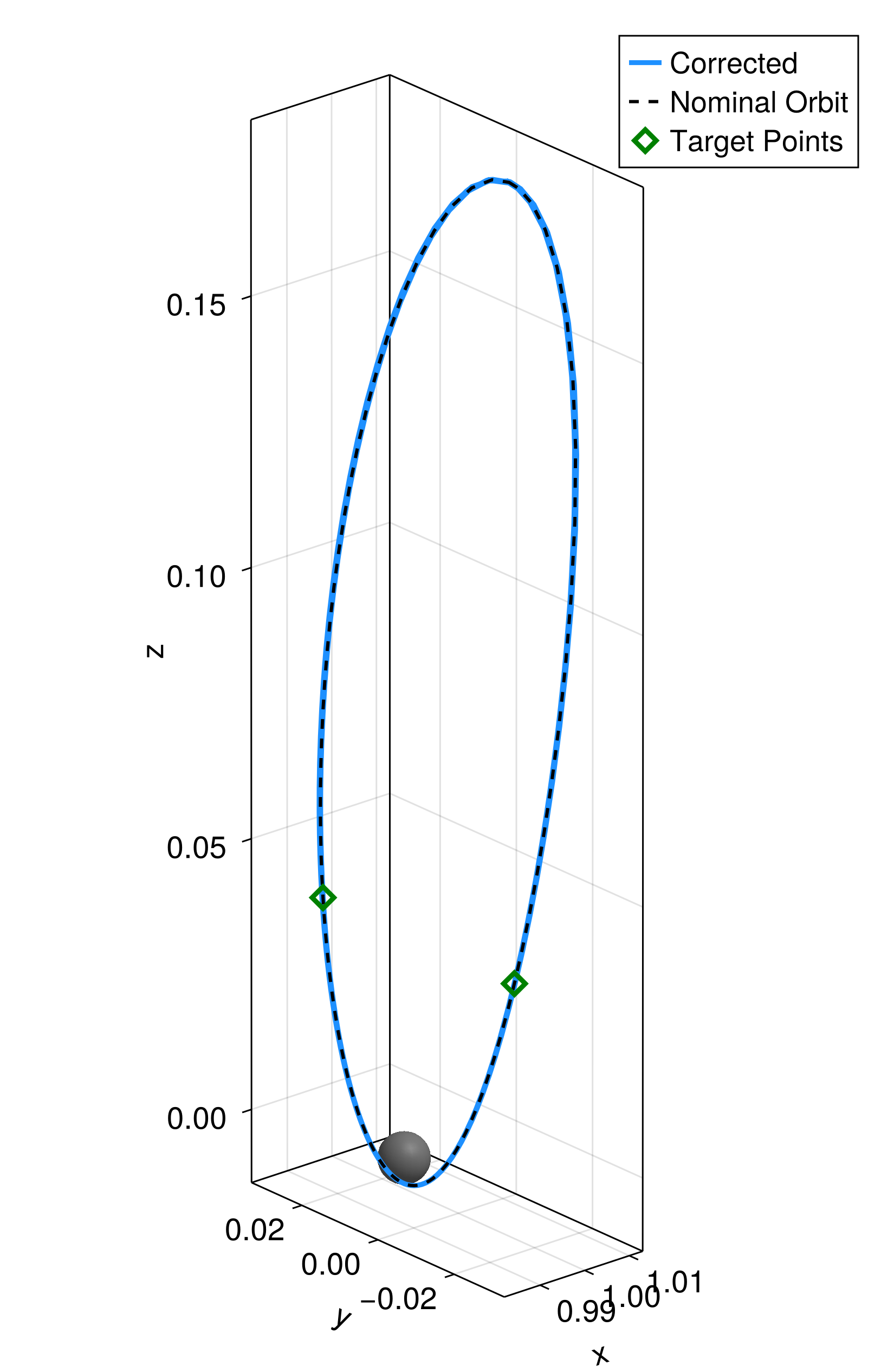}}
    \hfill
    \subcaptionbox{Uncorrected trajectory showing the rapid destabilization from the \gls{nrho}.\label{subfig:LT_uncorrected_orbits}}{\includegraphics[width=0.48\textwidth]{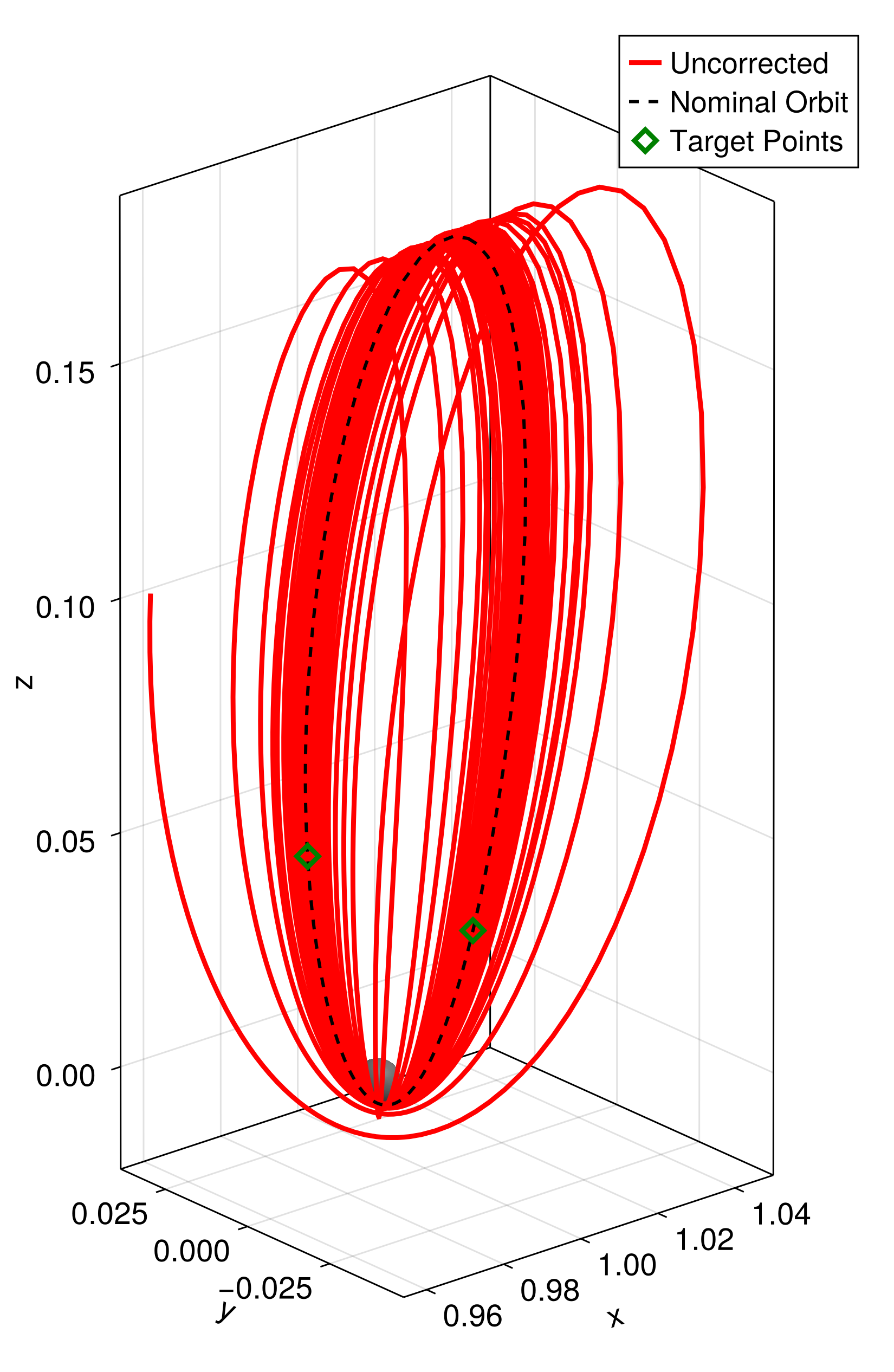}}%
    \hfill
    \caption{Visualizations of the low-thrust \gls{sk} scenario.}
    \label{fig:LT_scenario_plots}
\end{figure}

The \gls{tpa} to \gls{sk} in the \gls{cr3bp}, previously investigated in \cref{ss:impulsive_station_keeping}, is now revisited using the low-thrust formulation introduced in the preceding section. This section evaluates the robustness of the target-point \gls{sk} algorithm in stabilizing the nominal orbit against unmodeled dynamical perturbations. At every apolune passage, including at $t_0$, a perturbation is applied to the spacecraft state. Independent random unit vectors are generated for position and velocity, each sampled from a uniform spherical distribution. These directional vectors are subsequently scaled by randomized magnitudes drawn from zero-mean Gaussian distributions, using standard deviations of $\sigma=10$ km for position and $\sigma=0.3$ m/s for velocity. These errors are then added directly to the spacecraft state before the \gls{pop} is executed to solve for the subsequent orbital revolution. Ultimately, this framework enables an assessment of the algorithm's capacity to bound orbital deviations and maintain the spacecraft in close proximity to the nominal trajectory.

The spacecraft is assumed to have a starting mass of 1200 kg, a maximum thrust of $T_\textrm{max}=0.5$ N, and a specific impulse of $I_\textrm{sp}=3000$ s. Two target points are selected as those on the nominal orbit at 2.739 and 3.200 days. The control vector is defined within the spacecraft's \gls{rtn} frame. Consequently, although the thrust direction remains constant within this local frame during the thrust segment, its orientation continuously evolves relative to the Moon-centered reference frame.

Both the maximum order of the polynomial expansions and the maximum relaxation order for the moment-\gls{sos} hierarchy was set to 4. A simulation was ran for 62 revolutions of the nominal orbit, which coincided with just over one year total time at 368.2 days, thereby testing the ability of the optimization algorithm in long-term stabilization of the spacecraft under significant dynamical perturbations. The resulting trajectory for this entire period, shown in \cref{subfig:LT_corrected_orbits}, demonstrates the effectiveness of the \gls{sk} scheme in maintaining the spacecraft's vicinity to the nominal orbit. The fuel consumption for the total 368.2 day simulation was 0.56 kg. The minimum and mean altitude at periapsis was 30.65 km and 37.67 km, respectively, whilst the nominal altitude at periapsis was 36.25 km. For reference, the resulting trajectory without \gls{sk} is shown in \cref{subfig:LT_uncorrected_orbits}, demonstrating the sensitivity to state perturbations. However, it should be noted that the spacecraft descended below the Moon's surface on the $11^{\textrm{th}}$ revolution.

% ###############################################################################################
% ###############################################################################################

\section{Conclusions}\label{s:Conclusion}

This paper presented a methodology for solving spacecraft targeting and \gls{sk} problems utilizing polynomial and convex optimization techniques. By leveraging differential algebra to compute high-order Taylor expansions of the dynamics and constraints, the impulsive trajectory optimization problem was recast as a \gls{pop} and subsequently solved using moment-\gls{sos} optimization. For impulsive targeting scenarios, an efficient convex relaxation based on a second-order expansion was also derived. The moment-\gls{sos} formulation was shown to provide solutions as accurate as traditional \gls{nlp} solvers, but with the critical advantage of guaranteeing convergence to the global optimum under mild assumptions. Furthermore, the approach demonstrated robustness compared to standard linear approximations, particularly when handling large maneuvers and long propagation times. To demonstrate its versatility, the moment-\gls{sos} methodology was applied to a low-thrust \gls{sk} scenario in the Earth-Moon Circular Restricted Three-Body Problem. Significant state errors were simulated, validating the algorithm's capacity to bound orbital deviations against unmodeled dynamical perturbations. Ultimately, the ability to directly handle non-convex constraints and recast complex, nonlinear dynamics into formulations with reliable convergence properties without resorting the linear approximations makes the moment-\gls{sos} approach potentially suitable for autonomous onboard applications. Future work will focus on the extension of the proposed methods to multi-impulsive maneuver scenarios and to trajectory optimization problems with more complex objective functions and additional non-convex constraints, further highlighting the strengths of the \gls{pop} method in reliably finding the global optimum among several local minimizers.

\section*{Acknowledgment}

A portion of this material is based on research supported by the \gls{afosr} under award number: FA9550-23-1-0646.

\bibliographystyle{IEEEtran}
\bibliography{Journal/references}   % Use references.bib to resolve the labels.

\end{document}

%% file: Journal/acronyms.tex
% A
\newacronym{afosr}{AFOSR}{Air Force Office of Scientific Research}

% C
\newacronym{cr3bp}{CR3BP}{circular restricted three-body problem}

% D
\newacronym{da}{DA}{differential algebra}
\newacronym{dace}{DACE}{Differential Algebra Computational Engine}
\newacronym{dcp}{DCP}{disciplined convex programming}

% E
\newacronym[longplural={equations of motion}]{eom}{EOM}{equation of motion}

% G
\newacronym{gmp}{GMP}{Generalized Moment Problem}
\newacronym{gtpsa}{GTPSA}{Generalized Truncated Power Series Algebra}

% I
\newacronym{ic}{IC}{initial condition}
\newacronym{ipopt}{IPOPT}{Interior Point OPTimizer}

% L
\newacronym{lhs}{LHS}{left-hand side}
\newacronym{lmi}{LMI}{linear matrix inequality}
\newacronym{lp}{LP}{linear programming}

% M
\newacronym{mc}{MC}{Monte Carlo}
\newacronym{mop}{MOP}{moment optimization problem}

% N
\newacronym{nlp}{NLP}{nonlinear programming}
\newacronym{nrho}{NRHO}{near rectilinear halo orbit}

% P
\newacronym{pop}{POP}{polynomial optimization problem}
\newacronym{psd}{PSD}{positive semidefinite}

% R
\newacronym{rhs}{RHS}{right-hand side}
\newacronym{rk}{RK}{Runge-Kutta}
\newacronym{r2bp}{R2BP}{restricted two-body problem}
\newacronym{rtn}{RTN}{radial-transverse-normal}

% S
\newacronym{sdp}{SDP}{semidefinite programming}
\newacronym{sk}{SK}{station keeping}
\newacronym{sos}{SOS}{sum-of-squares}
\newacronym{stm}{STM}{state transition matrix}

% T
\newacronym{tpa}{TPA}{target point approach}

%% file: Journal/references.bib
@book{Ben-Tal2001,
  title = {Lectures on {{Modern Convex Optimization}}: {{Analysis}}, {{Algorithms}}, and {{Engineering Applications}}},
  shorttitle = {Lectures on {{Modern Convex Optimization}}},
  author = {{Ben-Tal}, Aharon and Nemirovski, Arkadi},
  year = {2001},
  series = {{{MPS-SIAM}} Series on Optimization},
  publisher = {{Society for Industrial and Applied Mathematics (SIAM)}},
  address = {Philadelphia, PA},
  doi = {10.1137/1.9780898718829},
  url = {http://epubs.siam.org/doi/book/10.1137/1.9780898718829},
  urldate = {2025-05-08},
  isbn = {978-0-89871-491-3 978-0-89871-882-9},
  langid = {english}
}

@book{Berz1999,
  title = {Modern {{Map Methods}} in {{Particle Beam Physics}}},
  author = {Berz, Martin},
  year = {1999},
  publisher = {Academic Press},
  address = {London, UK},
  url = {https://www.bmtdynamics.org/pub/papers/AIEP108book/AIEP108book.pdf},
  isbn = {0-12-014750-5}
}

@book{Boyd2004,
  title = {Convex {{Optimization}}},
  author = {Boyd, Stephen and Vandenberghe, Lieven},
  year = {2004},
  month = mar,
  edition = {1},
  publisher = {Cambridge University Press},
  address = {Cambridge, UK},
  doi = {10.1017/CBO9780511804441},
  url = {https://www.cambridge.org/core/product/identifier/9780511804441/type/book},
  urldate = {2024-07-03},
  copyright = {https://www.cambridge.org/core/terms},
  isbn = {978-0-521-83378-3 978-0-511-80444-1}
}

@article{Coey2022,
  title = {Solving {{Natural Conic Formulations}} with {{Hypatia}}.jl},
  author = {Coey, Chris and Kapelevich, Lea and Vielma, Juan Pablo},
  year = {2022},
  month = sep,
  journal = {INFORMS Journal on Computing},
  volume = {34},
  number = {5},
  pages = {2686--2699},
  publisher = {INFORMS},
  issn = {1091-9856},
  doi = {10.1287/ijoc.2022.1202},
  url = {https://pubsonline.informs.org/doi/abs/10.1287/ijoc.2022.1202},
  urldate = {2025-05-08}
}

@inproceedings{Deniau2015,
  title = {Generalised {{Truncated Power Series Algebra}} for {{Fast Particle Accelerator Transport Maps}}},
  booktitle = {Proc. 6th {{International Particle Accelerator Conference}} ({{IPAC}}'15), {{Richmond}}, {{VA}}, {{USA}}, {{May}} 3-8, 2015},
  author = {Deniau, Laurent and Tomoiag{\u a}, Ciprian},
  year = {2015},
  month = jun,
  series = {International {{Particle Accelerator Conference}}},
  pages = {374--377},
  publisher = {JACoW},
  address = {Geneva, Switzerland},
  doi = {10.18429/JACoW-IPAC2015-MOPJE039},
  url = {https://accelconf.web.cern.ch/IPAC2015/doi/JACoW-IPAC2015-MOPJE039.html},
  urldate = {2025-05-26},
  isbn = {978-3-95450-168-7},
  langid = {english}
}

@misc{Dixit2023,
  title = {Optimization.jl: {{A Unified Optimization Package}}},
  shorttitle = {Optimization.jl},
  author = {Dixit, Vaibhav Kumar and Rackauckas, Christopher},
  year = {2023},
  month = mar,
  doi = {10.5281/zenodo.7738525},
  url = {https://zenodo.org/records/7738525},
  urldate = {2025-04-21},
  howpublished = {Zenodo}
}

@inproceedings{Fu2020,
  title = {A {{High-order Target Point Approach}} to the {{Stationkeeping}} of {{Near Rectilinear Halo Orbits}}},
  booktitle = {71th {{International Astronautical Congress}}, {{CyberSpace Edition}}},
  author = {Fu, Xiaoyu and Baresi, Nicola and Armellin, Roberto},
  year = {2020},
  month = oct,
  pages = {1--12},
  address = {Virtual},
  url = {https://openresearch.surrey.ac.uk/esploro/outputs/conferenceProceeding/A-High-order-Target-Point-Approach-to/99521523802346},
  urldate = {2025-04-29},
  langid = {english}
}

@incollection{Henrion2005,
  title = {Detecting {{Global Optimality}} and {{Extracting Solutions}} in {{GloptiPoly}}},
  booktitle = {Positive {{Polynomials}} in {{Control}}},
  author = {Henrion, Didier and Lasserre, Jean-Bernard},
  editor = {Henrion, Didier and Garulli, Andrea},
  year = {2005},
  series = {Lecture {{Notes}} in {{Control}} and {{Information Science}}},
  volume = {312},
  pages = {293--310},
  publisher = {Springer},
  address = {Berlin, Heidelberg},
  doi = {10.1007/10997703\_15},
  url = {https://doi.org/10.1007/10997703\_15},
  urldate = {2025-04-21},
  isbn = {978-3-540-31594-0},
  langid = {english}
}

@misc{Henrion2008,
  title = {{{GloptiPoly}} 3: Moments, Optimization and Semidefinite Programming},
  shorttitle = {{{GloptiPoly}} 3},
  author = {Henrion, Didier and Lasserre, Jean-Bernard and L{\"o}fberg, Johan},
  year = {2008},
  month = sep,
  url = {https://homepages.laas.fr/henrion/software/gloptipoly3/}
}

@book{Henrion2023,
  author = {Henrion, Didier},
  title = {Moments for polynomial optimization - An illustrated tutorial},
  publisher = {Recent Trends in Computer Algebra, Institut Henri Poincar\'e, Paris},
  year = {2023},
  url = {https://homepages.laas.fr/henrion/papers/moments.pdf}
}

@article{Howell1993,
  title = {Station-Keeping Method for Libration Point Trajectories},
  author = {Howell, K. C. and Pernicka, H. J.},
  year = {1993},
  month = jan,
  journal = {Journal of Guidance, Control, and Dynamics},
  volume = {16},
  number = {1},
  pages = {151--159},
  publisher = {{American Institute of Aeronautics and Astronautics}},
  issn = {0731-5090},
  doi = {10.2514/3.11440},
  url = {https://doi.org/10.2514/3.11440},
  urldate = {2025-04-29}
}

@article{howell1995station,
  title={Station-keeping strategies for libration point orbits- Target point and Floquet Mode approaches},
  author={Howell, Kathleen C and Keeter, Timothy M},
  journal={Spaceflight mechanics 1995},
  pages={1377--1396},
  year={1995}
}

@article{gomez1998station,
  title={Station-keeping strategies for translunar libration point orbits},
  author={G{\'o}mez, Gerard and Howell, Kathleen C and Masdemont, Josep and Sim{\'o}, Carles},
  journal={Advances in Astronautical Sciences},
  volume={99},
  number={2},
  pages={949--967},
  year={1998}
}

@misc{HSL2023,
  title = {{HSL. A collection of Fortran codes for large scale scientific computation}},
  year = {2023},
  url = {http://www.hsl.rl.ac.uk/},
  urldate = {2025-04-18}
}

@article{Lasserre2001,
  title = {Global {{Optimization}} with {{Polynomials}} and the {{Problem}} of {{Moments}}},
  author = {Lasserre, Jean B.},
  year = {2001},
  month = jan,
  journal = {SIAM Journal on Optimization},
  volume = {11},
  number = {3},
  pages = {796--817},
  publisher = {{Society for Industrial and Applied Mathematics}},
  issn = {1052-6234},
  doi = {10.1137/S1052623400366802},
  url = {https://epubs.siam.org/doi/abs/10.1137/S1052623400366802},
  urldate = {2025-04-17}
}

@book{Lasserre2009,
  title = {Moments, {{Positive Polynomials And Their Applications}}},
  author = {Lasserre, Jean Bernard},
  year = {2009},
  series = {Imperial {{College Press Optimization Series}}},
  edition = {1},
  volume = {1},
  publisher = {Imperial College Press},
  address = {London, UK},
  isbn = {978-1-84816-445-1}
}

@inproceedings{Legat2017,
  title = {Sum-of-Squares Optimization in {{Julia}}},
  booktitle = {{{JuMP Developers Meetup}}/{{Workshop}}},
  author = {Legat, Beno{\^i}t and Coey, Christopher and Deits, Robin and Huchette, Joey and Perry, Amelia},
  year = {2017},
  month = jun,
  address = {Cambridge, MA},
  url = {https://hdl.handle.net/2078.1/195571},
  urldate = {2025-04-21},
  langid = {english}
}

@misc{Mosek2024,
  title = {{{MOSEK Modeling Cookbook}}},
  author = {{MOSEK ApS}},
  year = {2024},
  month = sep,
  pages = {1--124},
  institution = {MOSEK ApS},
  url = {http://docs.mosek.com/MOSEKModelingCookbook-a4paper.pdf}
}

@misc{Mosek2025,
  author = {{MOSEK ApS}},
  title = {{MOSEK version 11.0}},
  year = {2025},
  url = {https://www.mosek.com/},
  urldate = {2025-04-18}
}

@article{Rackauckas2017,
  title = {{{DifferentialEquations}}.jl -- {{A Performant}} and {{Feature-Rich Ecosystem}} for {{Solving Differential Equations}} in {{Julia}}},
  author = {Rackauckas, Christopher and Nie, Qing},
  year = {2017},
  month = may,
  journal = {Journal of Open Research Software},
  volume = {5},
  number = {1},
  pages = {1--10},
  issn = {2049-9647},
  doi = {10.5334/jors.151},
  url = {https://openresearchsoftware.metajnl.com/articles/10.5334/jors.151},
  urldate = {2025-04-21},
  langid = {american}
}

@inproceedings{Udell2014,
  title = {Convex {{Optimization}} in {{Julia}}},
  booktitle = {2014 {{First Workshop}} for {{High Performance Technical Computing}} in {{Dynamic Languages}}},
  author = {Udell, Madeleine and Mohan, Karanveer and Zeng, David and Hong, Jenny and Diamond, Steven and Boyd, Stephen},
  year = {2014},
  month = nov,
  pages = {18--28},
  publisher = {IEEE},
  address = {New Orleans, LA},
  doi = {10.1109/HPTCDL.2014.5},
  url = {https://ieeexplore.ieee.org/document/7069900},
  urldate = {2025-04-21},
  isbn = {978-1-4799-7020-9}
}

@article{Wachter2006,
  title = {On the Implementation of an Interior-Point Filter Line-Search Algorithm for Large-Scale Nonlinear Programming},
  author = {W{\"a}chter, Andreas and Biegler, Lorenz T.},
  year = {2006},
  month = mar,
  journal = {Mathematical Programming},
  volume = {106},
  number = {1},
  pages = {25--57},
  publisher = {Springer-Verlag},
  issn = {1436-4646},
  doi = {10.1007/s10107-004-0559-y},
  url = {https://link.springer.com/article/10.1007/s10107-004-0559-y},
  urldate = {2025-03-30},
  copyright = {2005 Springer-Verlag Berlin Heidelberg},
  langid = {english}
}

@inproceedings{Weisser2019,
  title = {Polynomial and {{Moment Optimization}} in {{Julia}} and {{JuMP}}},
  booktitle = {{{JuliaCon}}},
  author = {Weisser, Tillman and Legat, Beno{\^i}t and Coey, Chris and Kapelevich, Lea and Vielma, Juan Pablo},
  year = {2019},
  month = jul,
  address = {Baltimore, MD},
  url = {https://pretalx.com/juliacon2019/talk/QZBKAU/}
}

@article{Evans2024_1,
author = {Evans, Adam and Armellin, Roberto and Pirovano, Laura and Baresi, Nicola},
title = {High-Order Guidance for Time-Optimal Low-Thrust Trajectories with Accuracy Control},
journal = {Journal of Guidance, Control, and Dynamics},
volume = {47},
number = {2},
pages = {279-290},
year = {2024},
doi = {10.2514/1.G007540}
}

@article{Evans2024_2,
author = {Evans, Adam and Armellin, Roberto and Pirovano, Laura},
title = {Low-Thrust Fuel-Optimal Guidance with Automatic Control Sequence Detection and Separation},
journal = {Journal of Guidance, Control, and Dynamics},
volume = {47},
number = {12},
pages = {2512-2524},
year = {2024},
doi = {10.2514/1.G008093}
}

@inproceedings{Mao2016,
author = {Mao, Yuanqi and Szmuk, Michael and Açıkmeşe, Behçet},
year = {2016},
month = {12},
pages = {3636-3641},
title = {Successive convexification of non-convex optimal control problems and its convergence properties},
doi = {10.1109/CDC.2016.7798816}
}

@techreport{farquhar1970control,
  author      = {Farquhar, Robert W.},
  title       = {The Control and Use of Libration-Point Satellites},
  institution = {NASA Goddard Space Flight Center},
  year        = {1970},
  number      = {NASA TR R-346},
  month       = {September},
  url         = {https://ntrs.nasa.gov/citations/19700030541}
}

@incollection{PicardIteration,
author={Derks, Gianne},
editor={Meyers, Robert A.},
title={Existence and Uniqueness of Solutions of Initial Value Problems},
booktitle={Mathematics of Complexity and Dynamical Systems},
year={2011},
publisher={Springer New York},
pages={383-394},
doi={10.1007/978-1-4614-1806-1_25}
}

@incollection{PicardBanach,
author = {Berinde, Vasile},
title={The Picard Iteration},
booktitle={Iterative Approximation of Fixed Points},
year={2007},
publisher={Springer Berlin Heidelberg},
pages={31-62},
doi={10.1007/978-3-540-72234-2_2}
}
